\documentclass[10pt]{article}
\usepackage{amsmath, geometry, amssymb, amsthm, latexsym}
\usepackage{graphicx}

\def\N{\mathbb N}

\def\R{\mathbb R}

\def\Irr{\rm Irr}

\def\bs{\mathbb{s}}

\def\bA{\mathbb{A}}

\def\bC{\mathbb{C}}

\def\bF{\mathbb{F}}
\def\bG{\mathbb{G}}

\def\bQ{\mathbb{Q}}
\def\bR{\mathbb{R}}

\def\bT{\mathbb{T}}

\def\bZ{\mathbb{Z}}
\def\cA{\mathcal{A}}
\def\cB{\mathcal{B}}

\def\cD{\mathcal{D}}
\def\cE{\mathcal{E}}

\def\cH{\mathcal{H}}
\def\cI{\mathcal{I}}
\def\cJ{\mathcal{J}}
\def\cK{\mathcal{K}}

\def\cM{\mathcal{M}}

\def\cO{\mathcal{O}}

\def\cS{\mathcal{S}}

\def\cW{\mathcal{W}}
\def\cX{\mathcal{X}}

\def\cZ{\mathcal{Z}}

\def\fg{\mathfrak{g}}

\def\fl{\mathfrak{l}}

\def\fz{\mathfrak{z}}

\def\fB{\mathfrak{B}}

\def\fF{\mathfrak{F}}

\def\fX{\mathfrak{X}}

\def\Aut{\operatorname{Aut}}
\def\BC{\operatorname{BC}}

\def\End{\operatorname{End}}

\def\Frob{\operatorname{Frob}}
\def\GL{\operatorname{GL}}

\def\Gal{\operatorname{Gal}}
\def\Hasse{\operatorname{Hasse}}

\def\Hom{\operatorname{Hom}}

\def\Im{\operatorname{Im}}
\def\Ind{\operatorname{Ind}}

\def\Lie{\operatorname{Lie}}

\def\RS{\operatorname{RS}}

\def\Spec{\operatorname{Spec}}
\def\Sp{\operatorname{Sp}}
\def\Stab{\operatorname{Stab}}

\def\WD{\operatorname{WD}}
\def\ad{\operatorname{ad}}
\def\can{\operatorname{can}}

\def\cl{\operatorname{cl}}
\def\del{\partial}
\def\det{\operatorname{det}}

\def\dim{\operatorname{dim}}

\def\Disc{\operatorname{Disc}}

\def\exp{\operatorname{exp}}

\def\herm{\operatorname{herm}}
\def\id{\operatorname{id}}
\def\im{\operatorname{im}}
\def\ker{\operatorname{ker}}

\def\lin{\operatorname{lin}}
\def\log{\operatorname{log}}
\def\min{\operatorname{min}}
\def\mod{\operatorname{mod}}

\def\ord{\operatorname{ord}}
\def\rk{\operatorname{rk}}

\def\spl{\operatorname{spl}}

\def\sub{\operatorname{sub}}
\def\Std{\operatorname{Std}}
\def\tor{\operatorname{tor}}
\def\tr{\operatorname{tr}}
\def\rec{\operatorname{rec}}
\def\univ{\operatorname{univ}}

\def\Art{\operatorname{Art}}

\def\rec{\operatorname{rec}}

\def\bs{\backslash}

\newcommand{\mat}[4]{\left( \begin{array}{cc} {#1} & {#2} \\ {#3} & {#4}
\end{array} \right)}
\newcommand{\beas}{\begin{eqnarray*}}
\newcommand{\eeas}{\end{eqnarray*}}

\newcommand{\hooksearrow}{\mathrel{\rotatebox[origin=t]{-45}{$\hookrightarrow$}}}

\newcommand{\surjects}{\relbar\joinrel\twoheadrightarrow}
\newcommand{\isom}{\stackrel{\sim}{\longrightarrow}}

\theoremstyle{plain}
\newtheorem{theorem}{Theorem}[section]
\newtheorem*{theorem*}{Theorem}
\newtheorem{lemma}[theorem]{Lemma}
\newtheorem{corollary}[theorem]{Corollary}
\newtheorem{proposition}[theorem]{Proposition}

\theoremstyle{definition}
\newtheorem{definition}{Definition}

\title{Local-global compatibility for regular algebraic cuspidal automorphic representations when $\ell \neq p$}
\author{Ila Varma}

\begin{document}
\maketitle

\begin{abstract}
We prove the compatibility of local and global Langlands correspondences for $\GL_n$ up to semisimplification for the Galois representations constructed in \cite{HLTT,Scholze}. More precisely, let $r_p(\pi)$ denote an $n$-dimensional $p$-adic representation of the Galois group of a CM field $F$ attached to a regular algebraic cuspidal automorphic representation $\pi$ of $\GL_n(\bA_F)$. We show that the restriction of $r_p(\pi)$ to the decomposition group of a place $v\nmid p$ of $F$  corresponds up to semisimplification to $\rec(\pi_v)$, the image of $\pi_v$ under the local Langlands correspondence. Furthermore, we can show that the monodromy of the associated Weil-Deligne representation of $\left.r_p(\pi)\right|_{\Gal_{F_v}}$ is `more nilpotent' than the monodromy of $\rec(\pi_v)$.  
\end{abstract}

\section{Introduction}
Let $F$ be an imaginary CM (or totally real) field and let $\pi$ be a regular algebraic (i.e., $\pi_\infty$ has the same infinitesimal character as an irreducible algebraic representation $\rho_\pi$ of $\RS^F_\bQ \GL_n$) cuspidal automorphic representation of $\GL_n(\bA_F)$. In \cite{HLTT,Scholze}, the authors construct a continuous semisimple representation (depending on a choice of a rational prime $p$ and an isomorphism $\imath: \overline{\bQ}_p \isom \bC$)
	$$r_{p,\imath}(\pi): \Gal(\overline{F}/F) \longrightarrow \GL_n(\overline{\bQ}_p)$$
which satisfies the following: For every finite place $v \nmid p$ of $F$ such that $\pi$ and $F$ are both unramified at $v$, $r_{p,\imath}(\pi)$ is unramified at $v$ and
	$$\WD(\left.r_{p,\imath}(\pi)\right|_{G_{F_v}})^{ss} = \imath^{-1} \rec_{F_v}(\pi_v \otimes |\det|_v^{(1-n)/2})^{ss}.$$
Here $\rec_{F_v}$ as normalized in \cite{HT} denotes the local Langlands correspondence for $F_v$, and $\WD(r_v)$ denotes the Weil-Deligne representation associated to the the $p$-adic Galois representation $r_v$ of the decomposition group $G_{F_v} := \Gal(\overline{F_v}/F_v)$. 
In this paper, we extend local-global compatibility to all primes $v \nmid p$ of $F$. In particular, we prove the following theorem:

\begin{theorem*}[1$^{ss}$]\label{mainss} Keeping the notation of the previous paragraph, let $v \nmid p$ be a prime of $F$. Then
	$$\WD(\left.r_{p,\imath}(\pi)\right|_{G_{F_v}})^{ss} = \imath^{-1} \rec_{F_v}(\pi_v \otimes |\det|_v^{(1-n)/2})^{ss}.$$
\end{theorem*}

In fact, our methods allow us to `bound' the monodromy of $\WD(\left.r_{p,\imath}(\pi)\right|_{G_{F_v}})^{\Frob-ss}$ by the monodromy of $\rec_{F_v}(\pi_v|\det|_v^{(1-n)/2})$.  Using the notation introduced in Definition \ref{prec}, we can generalize the above theorem to:

\begin{theorem*}[1]\label{main} Keeping the notation from above, let $v \nmid p$ be a prime of $F$. Then
	$$\WD(\left.r_{p,\imath}(\pi)\right|_{G_{F_v}})^{\Frob-ss} \prec \imath^{-1} \rec_{F_v}(\pi_v|\det|_v^{(1-n)/2}),$$
where `$\Frob$-ss' denotes Frobenius semisimplification.
\end{theorem*}

To prove these results, we follow very similar arguments as in the construction of these Galois representations $r_{p,\imath}(\pi)$ as in \cite{HLTT}, which we will summarize below.

Let $\pi$ be a regular algebraic cuspidal automorphic representation on $\GL_n(\bA_F)$. Let $G$ denote the quasisplit unitary similitude group of signature $(n,n)$ associated to $F^{2n}$ and alternating form $\left(\begin{smallmatrix} 0 & 1_n \\-1_n & 0\end{smallmatrix}\right)$. It has a maximal parabolic $P = \{\GL_1 \times \left(\begin{smallmatrix}\ast & \ast \\0 & \ast\end{smallmatrix}\right)\} \subset G$ with Levi $L = \{\GL_1 \times\left(\begin{smallmatrix}\ast & 0 \\0 & \ast\end{smallmatrix}\right)\} \subset P$. However, note that $L \cong \GL_1 \times \RS^F_{\bQ}\GL_n$. For all sufficiently large positive integers $M$, let 
	$$\Pi(M) = \Ind_{P(\bA^{p,\infty})}^{G(\bA^{p,\infty})}(1 \times (\pi \otimes||\det||^M)^{p,\infty}).$$
The authors of \cite{HLTT} prove that $\Pi(M)$ is a subrepresentation of the space of overconvergent $p$-adic automorphic form on $G$ of some non-classical weight and finite slope. Classical cusp forms on this space base change via the trace formula to $\GL_{2n}$ to isobaric sums of conjugate self-dual cuspidal automorphic representations and they have Galois representations satisfying full local-global compatibility. Now, at all primes $v \nmid p$ of $F$ which split over $F^+$ (equivalently, at all primes away from $p$ where $G$ splits), the Bernstein centers associated to a finite union of Bernstein components containing $\Pi(M)_v$ map to the Hecke algebras acting on spaces of $p$-adic and classical cusp forms on $G$ of arbitrary integral (not necessarily classical) weight. For each $\sigma \in W_{F_v}$, the image of the Bernstein centers contains Hecke operators whose eigenvalue on a $p$-adic cusp form $\Pi'$ of $G$ is equal to 
	$$\tr \rec_{F_v} (\Pi'_v\otimes|\det|_v^{(1-2n)/2})(\sigma).$$ 
If $\pi'$ is classical, then local-global compatibility is already known, and so the eigenvalue is also equal to	
	$$\tr\WD(\left.r_p(\Pi')\right|_{G_{F_v}})^{ss}(\sigma),$$
where $r_p(\Pi'): G_F \rightarrow \GL_{2n}(\overline{\bQ}_p)$ denotes the Galois representation associated to $\Pi'$. By showing that there are linear combinations of classical cusp forms of $G$ whose Hecke eigenvalues are congruent mod $p^k$ to those of $\Pi(M)$ for each positive $k$, we are able to construct a continuous pseudorepresentation
	$T: G_F \rightarrow \overline{\bQ}_p$ satisfying the following: for every place $v \nmid p$ of $F$ which is split over $F^+$ and each $\sigma_v \in W_{F_v}$,
	$$T(\sigma_v) = \tr \rec_{F_v} (\Pi(M)_v\otimes|\det|_v^{(1-2n)/2})(\sigma_v).$$
This implies that there is a continuous semisimple Galois representation $r_{p,\imath}(\Pi(M)): G_F \rightarrow \GL_{2n}(\overline{\bQ}_p)$ whose trace is equal to $T$, and so for all primes $v$ of $F$ which are split over $F^+$ and lie above any rational prime other than $p$,
	$$\WD(\left.r_{p,\imath}(\Pi(M))\right|_{G_{F_v}})^{ss} \cong \imath^{-1} \rec_{F_v}(\Pi(M)_v \otimes |\det|_v^{(1-2n)/2})^{ss}.$$
At such primes $v$, we have 
	$$\Pi(M)_v = 1 \times (\pi_v \otimes |\det|_v^M) \boxplus (\pi_{{}^cv} \otimes |\det|_{{}^cv}^M)^{c,\vee},$$
thus if $\epsilon_p$ denotes the $p$-adic cyclotomic character, then $\WD(\left.r_{p,\imath}(\Pi(M)) \otimes \epsilon_p^{-M}\right|_{G_{F_v}})^{ss}$ is isomorphic to
	$$\imath^{-1} \rec_{F_v}(\pi_v|\det|_v^{(1-n)/2})^{ss} \oplus ((\imath^{-1} \rec_{F_{{}^cv}}(\pi_{{}^cv}|\det|^{(1-n)/2}_{{}^cv})^{ss})^{\vee,c}\otimes\epsilon_p^{1-2n-2M}).$$
Because we construct the above Galois representation for each sufficiently large positive integers $M$, it is now group theory to isolate an $n$-dimensional subquotient $r_{p,\imath}(\pi): G_F \rightarrow \GL_{n}(\overline{\bQ}_p)$ satisfying 
	$$\WD(\left.r_{p,\imath}(\pi)\right|_{G_{F_v}})^{ss} = \imath^{-1} \rec_{F_v}(\pi_v \otimes |\det|_v^{(1-n)/2})^{ss},$$
when $v \nmid p$ is a prime of $F$ which is split over $F^+$. Using the patching lemma by \cite{So}, we can remove the assumption that $v$ must split over $F^+$ and therefore conclude Theorem \ref{mainss}$^{ss}$. We then use idempotents constructed by \cite{SZ} and properties of $\wedge^k r_{p,\imath}(\Pi(M))$ and $\wedge^k \rec_{F_v}(\BC(\Pi(M))_v)$ to ``bound'' the monodromy of $\WD(r_{p,\imath}(\pi))^{\Frob-ss}$ by the monodromy of $\rec_{F_v}(\pi_v \otimes |\det|_v^{(1-n)/2})$.

\tableofcontents

\section{Notation and Conventions.}

In the sequel, $F$ will denote an imaginary CM field with maximal totally real subfield $F^+$, and $c$ will denote the nontrivial element of $\Gal(F/F^+)$. Let $p$ denote a rational prime such that all primes of $F^+$ above $p$ split in $F$. Let $n$ denote a positive integer and if $F^+ = \bQ$, assume $n > 2$. Let $\ell$ denote a rational prime such that $\ell \neq p$. Fix $\imath: \overline{\bQ}_p \isom \bC$. 

For any field $K$, $G_K$ will denote its absolute Galois group. If $K_0 \subset K$ is a subfield and $S$ is a finite set of primes of $K_0$, then we will denote by $G_K^S$ the maximal continuous quotient of $G_K$ in which all primes of $K$ not lying above an element of $S$ are unramified. 

If $K$ is an arbitrary number field and $v$ is a finite place of $K$, let $\varpi_v$ denote the uniformizer of $K_v$ and $k(v)$ is the residue field of $v$. Denote the absolute value on $K$ associated to $v$ by $|\cdot|_v$, which is normalized so that $|\varpi_{v}|_v = (\#k(v))^{-1}$. If $v$ is a real place of $K$, then $|x|_v := \pm x$, and if $v$ is complex, then $|x|_v = {}^cxx.$ Let 
	$$||\cdot||_K = \prod_{v} |\cdot|_v: \bA_K^\times \longrightarrow \bR^\times_{>0}.$$

If $r: G_{K_v} \rightarrow \GL_n(\overline{\bQ}_{p})$ denotes a continuous representation of $G_{K_v}$ where $v \nmid p$ is finite, then we will write $\WD(r)$ for the corresponding Weil-Deligne representation of the Weil group $W_{K_v}$ of $K_v$ (see section 1 of \cite{TY}). A Weil-Deligne representation is denoted as $(r,V,N) = (r,N) = (V,N)$, where $V$ is a finite-dimensional vector space over $\overline{\bQ}_p$, $r: W_{F_v} \rightarrow \GL(V)$ is a representation with open kernel and $N: V \rightarrow V$ is a nilpotent endomorphism with $r(\sigma)Nr(\sigma)^{-1} = |\Art_{F_v}^{-1}(\sigma)|_{F_v} N$ (here, $\Art_{F_v}: F_v^\times \stackrel{\sim}{\longrightarrow} W_{F_v}^{ab}$ denotes the local Artin map, normalized as in \cite{TY}). We say $(r,V,N)$ is Frobenius semisimple if $r$ is semisimple. We denote the Frobenius semisimplification of $(r,V,N)$ by $(r,V,N)^{\Frob-ss}$ and the semisimplification of $(r,V,N)$ is $(r,V,N)^{ss} = (r^{ss},V,0)$ (see Section 1 of \cite{TY}). 

If $\pi$ is an irreducible smooth representation of $\GL_n(K_v)$ over $\bC$, we will write $\rec_{K_v}(\pi)$ for the Weil-Deligne representation of $W_{K_v}$ corresponding to $\pi$ by the local Langlands correspondence (see \cite{HT}). If $\pi_1$ and $\pi_2$ are irreducible smooth representations of $\GL_{n_1}(K_v)$ (resp. $\GL_{n_2}(K_v)$), then there is an irreducible smooth representation $\pi_1 \boxplus \pi_2$ of $\GL_{n_1 + n_2}(K_v)$ over $\bC$ satisfying
	$$\rec_{K_v}(\pi_1 \boxplus \pi_2) = \rec_{K_v}(\pi_1) \oplus \rec_{F_v}(\pi_2).$$
	
Let $G$ be a reductive group over $K_v$, and let $P$ be a parabolic subgroup of $G$ with unipotent radical $N$ and Levi $L$. For a smooth representation $\pi$ of $L(K_v)$ on a vector space $V_{\pi}$ over a field $\Omega$ of characteristic 0, we define
	$\Ind_{P(K_v)}^{G(K_v)} \pi$ to be the representation of $G(K_v)$ by right translation on the set of locally constant functions $\varphi: G(K_v) \rightarrow V_{\pi}$ such that $\varphi(hg) = \pi(h)\varphi(g)$ for all $h \in P(F_v)$ and $g \in G(K_v)$. When $\Omega = \bC$, define normalized induction as
		$$n-\Ind_{P(K_v)}^{G(K_v)} \pi = \Ind_{P(K_v)}^{G(K_v)} \pi \otimes |\det(\ad\left.(h)\right|_{\Lie N})|_v^{1/2}.$$

\section{Recollections}

We recall the setup of \cite{HLTT}, including the unitary similitude group, the Shimura variety (and various compactifications) associated to the unitary group, their integral models. This will allow us to define automorphic vector bundles defined on these integral models, whose global sections will be the space of classical and $p$-adic automorphic forms.

\subsection{Unitary Group}
We define an integral unitary similitude group, which is associated to the following data. If $\Psi_n$ denotes the $n \times n$ matrix with $1$'s on the anti-diagonal and $0$'s elsewhere then let $J_n$ denote the element of $\GL_{2n}(\bZ)$ 
	$$J_n = \left(\begin{array}{cc}0 & \Psi_n \\-\Psi_n & 0\end{array}\right).$$
Let $\cD_F^{-1}$ denote the inverse different of $\cO_F$, and define the $2n$-dimensional lattice $\Lambda = (\cD_F^{-1})^{n} \oplus \cO_F^n$. Let $G$ be the group scheme over $\bZ$ defined by
	$$G(R) = \{(g,\mu) \in \Aut_{\cO_F \otimes_{\bZ} R}(\Lambda \otimes_{\bZ} R) \times R^\times : {}^tg J_n{}^cg = \mu J_n\}$$
for any ring $R$. Over $\bZ[1/\Disc(F/\bQ)]$, it is a quasi-split connected reductive group which splits over $\cO_{F^{nc}}[1/\Disc(F/\bQ)]$ where $F^{nc}$ denotes the Galois closure of $F/\bQ$. Let $\nu: G \rightarrow \GL_1$ be the multiplier character sending $(g,\mu) \mapsto \mu$. 

If $R = \Omega$ is an algebraically closed field of characteristic 0, then 
	$$G \times \Spec \Omega \cong \{ (\mu,g_\tau) \in \bG_m \times \GL_{2n}^{\Hom(F,\Omega)} : g_{\tau c} = \mu J_n {}^tg_{\tau}^{-1} J_n \quad \forall \tau \in \Hom(F,\Omega)\}.$$
	
Fix the lattice $\Lambda_{(n)} \cong (\cD_F^{-1})^n$ consisting of elements of $\Lambda$ whose last $n$ coordinates are equal to 0, and define $\Lambda_{(n)}' \cong \cO_F^n$ consisting of elements of $\Lambda$ whose first $n$ coordinates are equal to $0$.  Let $P^+_{(n)}$ denote the subgroup of $G$ preserving $\Lambda_{(n)}$. Write $L_{(n),\lin}$ for the subgroup of $P_{(n)}^+$ consisting of elements with $\nu = 1$ which preserve $\Lambda_{(n)}'$, and write $L_{(n),\herm}$ for the subgroup of $P_{(n)}^+$ which act trivially on $\Lambda/\Lambda_{(n)}$ and preserve $\Lambda_{(n)}'$. Then $L_{(n),\lin} \cong \RS^{\cO_F}_{\bZ} \GL_n$ and $L_{(n),\herm} \cong \bG_m$, and we can define $L_{(n)} := L_{(n),\lin} \times L_{(n),\herm}$.

Finally, let $G(\bA^\infty)^{\ord,\times} := G(\bA^{p,\infty}) \times P^+_{(n)}(\bZ_p)$, and $G(\bA^{\infty})^{\ord} = G(\bA^{p,\infty}) \times \varsigma_p^{\bZ_{\geq 0}} P_{(n)}^+(\bZ_p),$ where $\varsigma_p \in L_{(n),\herm}(\bQ_p) \cong \bQ_p^\times$ denotes the unique element with multiplier $p^{-1}$.

\subsection{Level Structure}
If $N_2 \geq N_1 \geq 0$ are integers, then let $U_p(N_1,N_2)$ be the subgroup of elements of $G(\bZ_p)$ which $\mod p^{N_2}$ lie in $P^+_{(n)}(\bZ/p^{N_2}\bZ)$ and map to 1 in $L_{(n),\lin}(\bZ/p^{N_1}\bZ)$. If $U^p$ is an open compact subgroup of $G(\bA^{p,\infty})$ we write $U^p(N_1, N_2)$ for $U^p \times U_p(N_1,N_2)$. 

If $N \geq 0$ is an integer we write $U_p(N)$ for the kernel of the map $P^{+}_{(n)}(\bZ_p) \rightarrow L_{(n),\lin}(\bZ/p^N\bZ)$. $U_p(N)$ will also denote the image of this kernel inside $L_{(n),\lin}(\bZ_p)$. 

\subsection{Shimura Variety}

Fix a neat open compact subgroup $U$ (as defined in section 0.6 of \cite{Pi}), and let $S$ be a locally noetherian scheme over $\bQ$. Recall from \S3.1 in \cite{HLTT} that a polarized $G$-abelian scheme with $U$-level structure is an abelian scheme $A$ over $S$ of relative dimension $n[F:\bQ]$ along with the following data:
\begin{itemize}
\item An embedding $\imath: F \hookrightarrow \End^0(A)$ such that $\Lie A$ is a locally free of rank $n$ over $F \otimes_{\bQ} \cO_S$. 
\item A polarization $\lambda: A \rightarrow A^\vee$ 
\item $U$-level structure $[\eta]$.
\end{itemize} 
For more precise definitions, see Pg. 76-77 of \cite{HLTT}. Denote by $X_{U}$ the smooth quasi-projective scheme over $\bQ$ which represents the functor that sends a locally noetherian scheme $S/\bQ$ to the set of quasi-isogeny classes of polarized $G$-abelian schemes with $U$-level structure. Let $[(A^{\univ},\imath^{\univ},\lambda^{\univ},[\eta^{\univ}])]$ denote the universal equivalence class of polarized $G$-abelian varieties with $U$-level structure. Allowing $U$ to vary, the inverse system $\{X_{U}\}$ has a right $G(\bA^\infty)$ action, with finite etale transition maps $g: X_{U} \rightarrow X_{U'}$ whenever $U' \supset g^{-1}Ug$. 

For each $U$, denote by $\Omega^1_{A^{\univ}/X_{U}}$ the sheaf of relative differentials on $A^{\univ}$. Let $\Omega_U$ denote the Hodge bundle, i.e. the pullback by the identity section of $\Omega^1_{A^{\univ}/X_U}$. It is locally free of rank $n[F:\bQ]$ and does not depend on $A^{\univ}$.

For each neat open compact subgroup $U \subset G(\bA^\infty)$, there is a normal projective scheme $X^{\min}_{U}$ over $\Spec \bQ$ together with a $G(\bA^\infty)$-equivariant dense open embedding
	$$j_U: X_U \hookrightarrow X^{\min}_U$$
which is known as the minimal compactification of $X_U$. Let the boundary be denoted by $\del X^{\min}_U = X^{\min}_U - X_{U}$. The inverse system $\{X^{\min}_U\}$ also has a right $G(\bA^\infty)$-action. Furthermore, there is a normal projective flat $\bZ_{(p)}$ scheme $\cX^{\min}_{U}$ whose generic fiber is $X^{\min}_U$. We will denote the ample line bundle on $\cX^{\min}_{U}$ by $\omega_U$; by Propositions 2.2.1.2 and 2.2.3.1 in \cite{La4}, its pullback to $X_U$ is identified with $\wedge^{n[F:\bQ]}\Omega_U$. The system $\{\omega_U\}$ over $\{\cX_{U}^{\min}\}$ has an action of $G(\bA^{p,\infty} \times \bZ_p)$. If we let $\overline{X}_U^{\min} = \cX_U^{\min} \otimes_{\bZ_{(p)}} \bF_p$, then by Corollaries 6.3.1.7-8 in \cite{La4}, there is a canonical $G(\bA^{p,\infty})$-invariant section
$$\Hasse_U \in H^0(\overline{X}_U^{\min}, \omega_U^{\otimes (p-1)}) \quad \mbox{s.t.} \quad g^\ast \Hasse_{g^{-1}Ug} = \Hasse_U \quad \forall g \in G(\bA^{p,\infty} \times \bZ_p).$$

Denote by $\overline{X}^{\min n-\ord}_U$ the zero locus in $\overline{X}^{\min}_U$ of $\Hasse_U$. 
\begin{lemma}\label{affine} The non-zero locus $\overline{X}^{\min}_U \bs \overline{X}^{\min n-\ord}_U$ is relatively affine over $\overline{X}^{\min}_U$.
Furthermore, it is affine over $\bF_p$.
\end{lemma}
\begin{proof}
The nonzero locus over $\overline{X}^{\min}_U$ is associated to the sheaf of algebras $$\left(\oplus_{i=0}^\infty \omega_U^{\otimes (p-1)ai}\right)/(\Hasse^a_U - 1) \quad \forall a \in \bZ_{>0},$$
Over $\bF_p$, it is associated to the algebra
	$$\left(\oplus_{i=0}^\infty H^0(\overline{X}^{\min}_U,\omega^{\otimes(p-1)ai})\right)/(\Hasse_U^a - 1) \quad \forall a \in \bZ_{>0}.$$
\end{proof}

\subsection{Ordinary Locus}
Now let $\cS$ denote a scheme over $\bZ_{(p)}$ and fix a neat open compact subgroup $U^p$ along with two positive integers $N_2 \geq N_1$. Then the ordinary locus is a smooth quasi-projective scheme $\cX^{\ord}_{U^p(N_1,N_2)}$ over $\bZ_{(p)}$ representing the functor which sends $\cS$ to the the set of prime-to-$p$ quasi-isogeny classes of ordinary, prime-to-p quasi-polarized $G$-abelian schemes with $U^p(N_1,N_2)$-level structure as defined in \S3.1 of \cite{HLTT}. It is a partial integral model of $X_{U^p(N_1,N_2)}$. Let $[\cA^{\univ},\imath^{\univ},\lambda^{\univ},[\eta^{\univ}]]/\cX^{\ord}_{U^p(N_1,N_2)}$ denote the universal equivalence class of ordinary prime-to-$p$ quasi-polarized $G$-abelian schemes with $U^p(N_1,N_2)$-level structure up to quasi-isogeny. Finally denote by $\fX^{\ord}_{U^p(N_1,N_2)}$ the formal completion along the special fiber of $\cX^{\ord}_{U^p(N_1,N_2)}$, and $\overline{X}^{\ord}_{U^p(N_1,N_2)} = \cX^{\ord}_{U^p(N_1,N_2)} \otimes_{\bZ_{(p)}} \bF_p$. Both form inverse systems each with a right $G(\bA^\infty)^{\ord}$-action. Furthermore, the map
	$$\varsigma_p: \overline{X}^{\ord}_{U^p(N_1,N_2+1)} \rightarrow \overline{X}^{\ord}_{U^p(N_1,N_2)}$$
is the absolute Frobenius map composed with the forgetful map $\cX^{\ord}_{U^p(N_1,N_2+1)} \rightarrow \cX^{\ord}_{U^p(N_2,N_2)}$ for any $N_2 \geq N_1 \geq 0$. If $N_2 > 0$, then $\varsigma_p$ defines a finite flat map
	$$\varsigma_p: \cX^{\ord}_{U^p(N_1,N_2+1)} \rightarrow \cX^{\ord}_{U^p(N_1,N_2)}$$
with fibers of degree $p^{n^2[F^+:\bQ]}$ (see \S3.1 of \cite{HLTT}).

For each $U^p(N_1,N_2)$ such that $U^p$ is neat, there is a partial minimal compactification of the ordinary locus $\cX^{\ord}_{U^p(N_1,N_2)}$ denoted by $\cX^{\ord,\min}_{U^p(N_1,N_2)}$. By Theorem 6.2.1.1 in \cite{La4}, this compactification of the ordinary locus is a normal quasi-projective scheme over $\bZ_{(p)}$ together with a dense open $G(\bA^\infty)^{\ord}$-equivariant embedding
	$$j_{U^p(N_1,N_2)}: \cX^{\ord}_{U^p(N_1,N_2)} \hookrightarrow \cX^{\ord,\min}_{U^p(N_1,N_2)}.$$
Its generic fiber is $X^{\min}_{U^p(N_1,N_2)}$ but unlike $\cX^{\min}_{U^p(N_1,N_2)}$, it is not proper. Furthermore by Proposition 6.2.2.1 in \cite{La4}, the induced action of $g \in G(\bA^\infty)^{\ord}$ on $\{\cX^{\ord,\min}_{U^p(N_1,N_2)}\}$ is quasi-finite.
Write $\del\cX^{\ord,\min}_{U^p(N_1,N_2)} = \cX^{\ord,\min}_{U^p(N_1,N_2)} - \cX^{\ord}_{U^p(N_1,N_2)},$ for the boundary, and let $\fX^{\ord,\min}_{U^p(N_1,N_2)}$ be the formal completion along the special fiber of $\cX^{\ord,\min}_{U^p(N_1,N_2)}$.  Note that by Corollary 6.2.2.8 and Example 3.4.5.5 in \cite{La4}, the natural map
	$$\fX^{\ord,\min}_{U^p(N_1,N_2')} \isom \fX^{\ord,\min}_{U^p(N_1,N_2)}$$
is an isomorphism, and so we will drop $N_2$ from notation. Define $\overline{X}^{\ord,\min}_{U^p(N_1,N_2)} = \cX^{\ord,\min}_{U^p(N_1,N_2)} \otimes_{\bZ_{(p)}} \bF_p$.

For each $U^p(N_1,N_2)$ note that there is a $G(\bA^\infty)^{\ord,\times}$-equivariant open embeddings $\cX^{\ord,\min}_{U^p(N_1,N_2)} \hookrightarrow \cX^{\min}_{U^p(N_1,N_2)}$. This induces a map on the special fibers
	$$\overline{X}^{\ord,\min}_{U^p(N_1,N_2)} \hookrightarrow \overline{X}^{\min}_{U^p(N_1,N_2)} \bs \overline{X}^{\min,n-\ord}_{U^p(N_1,N_2)},$$
which is both an open and closed embedding by Proposition 6.3.2.2 of \cite{La4}.
Note that only when the level is prime-to-$p$ is the nonzero locus of $\Hasse_{U^p(N_1,N_2)}$ isomorphic to the special fiber of the minimally compactified ordinary locus. When $N_2 > 0$, the above map is not an isomorphism. 

\subsection{Toroidal compactifications}

We now introduce toroidal compactifications of $X_U$ and $\cX_{U^p(N_1,N_2)}$ which are parametrized by neat open compact subgroups of $G(\bA^\infty)$ and certain cone decompositions defined in \cite{La4} and \cite{HLTT}. Let $\cJ^{\tor}$ be the indexing set of pairs $(U,\Delta)$ defined in Proposition 7.1.1.21 in \cite{La4} or Pg. 156 in \cite{HLTT}, where $U$ is a neat open compact subgroup and $\Delta$ is a $U$-admissible cone decomposition as defined in \S5.2 \cite{HLTT}. We will not recall the definition here as it is not necessary for any argument. 

If $(U,\Delta) \in \cJ^{\tor}$ then by Theorem 1.3.3.15 of \cite{La4}, there is a smooth projective scheme $X_{U,\Delta}$ along with a dense open embedding
	$$j_{U,\Delta}: X_{U} \hookrightarrow X_{U,\Delta}$$
and a projection $\pi_{\tor/\min}: X_{U,\Delta} \rightarrow X^{\min}_U$ such that the following diagram commutes:
\begin{eqnarray*} X_U &\hookrightarrow& X_{U,\Sigma} \\
& \large \hooksearrow& \downarrow \\
& & X_U^{\min}.
\end{eqnarray*}
Furthermore, the boundary $\del X_{U,\Delta} = X_{U,\Delta} \bs j_{U,\Delta} X_{U,\Delta}$ is a divisor with simple normal crossings. The collection $\{X_{U,\Delta}\}_{\cJ^{\tor}}$ becomes a system of schemes with a right $G(\bA^\infty)$-action via the maps $\pi_{(U,\Delta)/(U',\Delta')}: X_{U,\Delta} \rightarrow X_{U',\Delta'}$ whenever $(U,\Delta) \geq (U',\Delta')$ (see Pg. 153 of \cite{HLTT} for the definition of $\geq$ in this context).

If $(U^p(N_1,N_2),\Delta) \in \cJ^{\tor}$, then by Theorem 7.1.4.1 of \cite{La4}, there is a smooth quasi-projective scheme $\cX^{\ord}_{U^p(N_1,N_2),\Delta}$ along with a dense open embedding
	$$j^{\ord}_{U^p(N_1,N_2),\Delta}: \cX^{\ord}_{U^p(N_1,N_2)}\hookrightarrow \cX^{\ord}_{U^p(N_1,N_2),\Delta}$$
and a projection
	$$\pi^{\ord}_{\tor/\min}: \cX^{\ord}_{U^p(N_1,N_2),\Delta} \rightarrow \cX^{\ord,\min}_{U^p(N_1,N_2)}$$
such that the following diagram commutes:
\begin{eqnarray*} \cX^{\ord}_{U^p(N_1,N_2)} &\hookrightarrow& \cX^{\ord}_{U^p(N_1,N_2),\Delta} \\
& \large \hooksearrow& \downarrow \\
& & \cX^{\ord,\min}_{U^p(N_1,N_2)}.
\end{eqnarray*}

Furthermore, the boundary $\del \cX^{\ord}_{U^p(N_1,N_2),\Delta} = \cX^{\ord}_{U^p(N_1,N_2),\Delta} \backslash j^{\ord}_{U^p(N_1,N_2),\Delta} \cX^{\ord}_{U^p(N_1,N_2),\Delta}$ is a divisor with simple normal crossings.  The collection $\{\cX^{\ord}_{U^p(N_1,N_2),\Delta}\}_{\cJ^{\tor}}$ becomes a system of schemes with a right $G(\bA^\infty)^{\ord}$-action via the maps $\pi_{(U^p(N_1,N_2),\Delta)/(U^{p'}(N_1',N_2'),\Delta')}: \cX^{\ord}_{U^p(N_1,N_2),\Delta} \rightarrow \cX^{\ord}_{U^{p'}(N_1',N_2),\Delta'}$ whenever $(U^p(N_1,N_2),\Delta) \geq (U^{p'}(N_1',N_2),\Delta')$ (see Pg. 153 of \cite{HLTT} for the definition of $\geq$ in this context).

\section{Automorphic Bundles}

We first define the coherent sheaves on $\cX^{\min}$ whose global sections are what we consider to be the finite part of classical cuspidal automorphic forms on $G$. They are locally free sheaves originally defined over the toroidal compactifications $X_{U,\Delta}$ and are then pushed forward to $X^{\min}$ via $\pi_{\tor,\min}$. We start by recalling some differential sheaves that have already been defined.

\subsection{Automorphic Bundles on compactifications of the Shimura Variety}

Recall from the previous section that we have a locally free sheaf $\Omega_U$ on $X_U$, which is the pullback by the identity  section of the sheaf of relative differentials from $A^{\univ}$, the universal abelian variety over $X_U$. On $\cX^{\min}_U$, the normal integral model of the minimal compactification of $X_U$, there is an ample line bundle $\omega_U$ whose pullback to $X_U$ is identified with $\wedge^{n[F:\bQ]} \Omega_U$.

Any universal abelian variety $A^{\univ}/X_{U}$ extends to a semi-abelian variety $A_\Delta/X_{U,\Delta}$ (see remarks 1.1.2.1 and 1.3.1.4 of \cite{La4}). Define $\Omega_{U,\Delta}$ as the pullback by the identity section of the sheaf of relative differentials on $A_{\Delta}$. Note that when restricting to the Shimura variety $X_U$, $\left.\Omega_{U,\Delta}\right|_{X_U}$ is canonically isomorphic to $\Omega_U$. Let $\cO_{X_{U,\Delta}}(||\nu||)$ denote the structure sheaf with $G(\bA^\infty)$-action twisted by $||\nu||$.

Let $\cE^{\can}_{U,\Delta}$ denote the principal $L_{(n)}$-bundle on $X_{U,\Delta}$ defined as follows: For a Zariski open $W$, $\cE^{\can}_{U,\Delta}(W)$ is the set of pairs of isomorphisms
	$$\xi_0: \left.\cO_{X_{U,\Delta}}(||\nu||)\right|_W \isom \cO_W \qquad \mbox{and} \qquad \xi_1: \Omega_{U,\Delta} \isom \Hom_{\bQ}(V/V_{(n)},\cO_W),$$
where $V = \Lambda \otimes \bQ = F^{2n}$ and $V_{(n)} = \Lambda_{(n)} \otimes \bQ \cong F^{n}$. There is an action of $h \in L_{(n)}$ on $\cE^{\can}_{U,\Delta}$ by
	$$h(\xi_0,\xi_1) = (\nu(h)^{-1} \xi_0, \xi_1 \circ h^{-1}).$$
The inverse system $\{\cE^{\can}_{U,\Delta}\}$ has an action of $G(\bA^\infty)$. 

Let $R$ be any irreducible noetherian $\bQ$-algebra. Fix a representation $\rho$ of $L_{(n)}$ on a finite, locally free $R$-module $W_{\rho}$. Define the locally free sheaf $\cE^{\can}_{U,\Delta,\rho}$ over $X_{U,\Delta} \times \Spec R$ as follows: For a Zariski open $W$, let $\cE^{\can}_{U,\Delta,\rho}(W)$ be the set of $L_{(n)}(\cO_W)$-equivariant maps of Zariski sheaves of sets
	$$\left.\cE^{\can}_{U,\Delta}\right|_W \rightarrow W_\rho \otimes_R \cO_W.$$
With fixed $\rho$, the system of sheaves $\{\cE^{\can}_{U,\Delta,\rho}\}$ has a $G(\bA^\infty)$-action. If $\Std$ denotes the representation over $\bZ$ of $L_{(n)}$ on $\Lambda/\Lambda_{(n)}$, then let $\omega_{U,\Delta} := \cE^{\can}_{U,\Delta,\wedge^{n[F:\bQ]}\Std^\vee}.$
We will write $\cI_{\del X_{U,\Delta}}$ for the ideal sheaf in $\cO_{X_{U,\Delta}}$ defining the boundary $\del X_{U,\Delta}$. Define the subcanonical extension
	$$\cE^{\sub}_{U,\Delta,\rho}  = \cE^{\can}_{U,\Delta,\rho} \otimes \cI_{\del X_{U,\Delta}}$$

Recall the projection $\pi_{\tor/\min}: X_{U,\Delta} \rightarrow X^{\min}_U$, and define $\cE^{\sub}_{U,\rho} = \pi_{\tor/\min\ast} \cE^{\sub}_{U,\Delta,\rho}.$  The coherent sheaves defined on $X^{\min}_U$ are independent of the choice of $\Delta$. If we fix $\rho$, there is an action of $G(\bA^\infty)$ on the system $\{\cE^{\sub}_{U,\rho}\}$ indexed by neat open compact subgroups.
 
Now let $\rho_0$ be a representation of $L_{(n)}$ on a finite locally free $\bZ_{(p)}$-module. By Definition 8.3.5.1 of \cite{La4}, there is a system of coherent sheaves associated to $\rho_0$ over $\{\cX^{\min}_U\}$ with $G(\bA^\infty)^{\ord,\times}$-action whose pull-back to $\{X^{\min}_U\}$ is $G(\bA^\infty)$-equivariantly identified with $\{\cE^{\sub}_{U,\rho_0 \otimes \bQ}\}$. We will also refer to these sheaves by $\cE^{\sub}_{U,\rho_0}$. Note that over $\cX^{\min}_U$,
	$$\cE^{\sub}_{U,\rho_0} \otimes \omega_{U} \cong \cE^{\sub}_{U,\rho_0 \otimes (\wedge^{n[F:\bQ]} \Std^\vee)},$$
where $\omega_U$ denotes the ample line bundle defined on $\cX^{\min}_U$.
\subsection{Automorphic Bundles on the Ordinary Locus}
We now define automorphic vector bundles on the system of integral models of the minimally compactified ordinary locus $\{\cX^{\ord,\min}_{U^p(N_1,N_2)}\}$ as well as its formal completion along the special fiber $\{\fX^{\ord,\min}_{U^p(N_1)}\}$. The global sections of these coherent sheaves will consist of what we consider cuspidal $p$-adic automorphic forms. We first recall some definitions of sheaves defined on $\cX^{\ord}_{U^p(N_1,N_2),\Delta}$.

Any universal abelian variety $\cA^{\univ}/\cX^{\ord}_{U^p(N_1,N_2)}$ extends uniquely to a semi-abelian variety $\cA_\Delta/\cX^{\ord}_{U^p(N_1,N_2),\Delta}$~by~Remarks 1.1.2.1 and 1.3.1.4 of \cite{La4}. Define $\Omega^{\ord}_{U^p(N_1,N_2),\Delta}$ as the pullback by the identity section of the sheaf of relative differentials on $\cA_{\Delta}$. The inverse system $\{\Omega^{\ord}_{U^p(N_1,N_2),\Delta}\}$ has an action of $G(\bA^\infty)^{\ord,\times}$. There is also a natural map
	$$\varsigma_p: \varsigma_p^{\ast}\Omega^{\ord}_{U^p(N_1,N_2-1),\Delta} \rightarrow \Omega^{\ord}_{U^p(N_1,N_2),\Delta}.$$

Denote by $\cO_{\cX^{\ord}_{U^p(N_1,N_2),\Delta}}(||\nu||)$ the structure sheaf $\cO_{\cX^{\ord}_{U^p(N_1,N_2),\Delta}}$ with $G(\bA^\infty)^{\ord}$-action twisted by $||\nu||$ (recall that $\{\cX^{\ord}_{U^p(N_1,N_2),\Delta}\}$ has a right $G(\bA^\infty)^{\ord}$-action).

Let $\cE^{\ord,\can}_{U^p(N_1,N_2),\Delta}$ denote the principal $L_{(n)}$-bundle on $\cX^{\ord}_{U^p(N_1,N_2),\Delta}$ in the Zariski topology defined as follows: For a Zariski open $W$, $\cE^{\ord,\can}_{U^p(N_1,N_2),\Delta}(W))$ is the set of pairs of isomorphisms
	$$\xi_0:\left.\cO_{\cX^{\ord}_{U^p(N_1,N_2),\Delta}}(||\nu||)\right|_W \isom \cO_W \qquad \mbox{and} \qquad \xi_1:\Omega^{\ord}_{U^p(N_1,N_2),\Delta} \isom \Hom_{\bZ}(\Lambda/\Lambda_{(n)}, \cO_W).$$
(Recall that $\Lambda_{(n)}$ is the sublattice of $\Lambda = (\cD_F^{-1})^n \oplus \cO_F^n$ consisting of elements whose last $n$ coordinates are equal to $0$.)
There is an action of $h \in L_{(n)}$ on $\cE^{\ord,\can}_{U^p(N_1,N_2),\Delta}$ by
	$$h(\xi_0,\xi_1) = (\nu(h)^{-1} \xi_0, \xi_1 \circ h^{-1} ).$$
The inverse system $\{\cE^{\ord,\can}_{U^p(N_1,N_2),\Delta}\}$ has an action of $G(\bA^\infty)^{\ord}$.
Let $R$ be an irreducible noetherian $\bZ_{(p)}$-algebra. Fix a representation $\rho$ of $L_{n,(n)}$ on a finite, locally free $R$-module $W_\rho$. Denote the canonical extension to $\cX^{\ord}_{U^p(N_1,N_2),\Delta,\rho} \times \Spec R$ of the automorphic vector bundle on $\cX^{\ord}_{U^p(N_1,N_2)}$ associated to $\rho$ by $\cE^{\ord,\can}_{U^p(N_1,N_2),\Delta,\rho}$, which is defined as follows: For any Zariski open $W$, $\cE^{\ord,\can}_{U^p(N_1,N_2),\Delta,\rho}(W)$ is the set of $L_{(n)}(\cO_W)$-equivariant maps of Zariski sheaves of sets
	$$\left.\cE^{\ord,\can}_{U^p(N_1,N_2),\Delta}\right|_W \rightarrow W_\rho \otimes_{R} \cO_W.$$
When $\rho$ is fixed, the system of sheaves $\{\cE^{\ord,\can}_{U^p(N_1,N_2),\Delta,\rho}\}$ has an action of $G(\bA^\infty)^{\ord}$.  Furthermore, the inverse of $\varsigma_p^{\ast}$ gives a map
	$$(\varsigma^{\ast}_p)^{-1}:{\varsigma_p}_{\ast} \cE^{\ord,\can}_{U^p(N_1,N_2),\Delta,\rho} \stackrel{\sim}{\longrightarrow} \cE^{\ord,\can}_{U^p(N_1,N_2-1),\Delta,\rho} \otimes_{\cO_{\cX^{\ord}_{U^p(N_1,N_2-1),\Delta}}} {\varsigma_{p}}_{\ast}\cO_{\cX^{\ord}_{U^p(N_1,N_2),\Delta}}.$$
Composing $(\varsigma_p^{\ast})^{-1}$ with $1 \otimes \tr_{\varsigma_p}: \cE^{\ord,\can}_{U^p(N_1,N_2 -1),\Delta,\rho} \otimes {\varsigma_p}_{\ast} \cO_{\cX^{\ord}_{U^p(N_1,N_2),\Delta}} \rightarrow \cE^{\ord,\can}_{U^p(N_1,N_2 -1),\Delta,\rho}$ gives a $G(\bA^\infty)^{\ord,\times}$-equivariant map
	$$\tr_F: {\varsigma_p}_{\ast} \cE^{\ord,\can}_{U^p(N_1,N_2),\Delta,\rho} \rightarrow \cE^{\ord,\can}_{U^p(N_1,N_2-1),\Delta,\rho}$$
satisfying $\tr_F \circ \varsigma_p^{\ast} = p^{n^2[F^+:\bQ]}$.
If $\Std$ denotes the representation over $\bZ$ of $L_{(n)}$ on $\Lambda/\Lambda_{(n)}$, then let $\omega_{U^p(N_1,N_2),\Delta} := \cE^{\ord,\can}_{U^p(N_1,N_2),\Delta,\wedge^{n[F:\bQ]}\Std^\vee}$ denote the pullback of $\omega_{U}$ to $\cX^{\ord}_{U^p(N_1,N_2),\Delta}$.
We will write $\cI_{\del \cX^{\ord}_{U^p(N_1,N_2),\Delta}}$ for the ideal sheaf in $\cO_{\cX^{\ord}_{U^p(N_1,N_2),\Delta}}$ defining the boundary $\del \cX^{\ord}_{U^p(N_1,N_2),\Delta}$. Define the subcanonical extension as
	$$\cE^{\ord,\sub}_{U^p(N_1,N_2),\Delta,\rho} = \cE^{\ord,\can}_{U^p(N_1,N_2),\Delta,\rho} \otimes \cI_{\del \cX^{\ord}_{U^p(N_1,N_2),\Delta}}.$$
Again, the inverse of $\varsigma_p^{\ast}$ gives a map
	$$(\varsigma^{\ast}_p)^{-1}:{\varsigma_p}_{\ast} \cE^{\ord,\can}_{U^p(N_1,N_2),\Delta,\rho} \stackrel{\sim}{\longrightarrow} \cE^{\ord,\can}_{U^p(N_1,N_2-1),\Delta,\rho} \otimes_{\cO_{\cX^{\ord}_{U^p(N_1,N_2-1),\Delta}}} {\varsigma_{p}}_{\ast}\cI_{\del\cX^{\ord}_{U^p(N_1,N_2),\Delta}}.$$
Composing $(\varsigma_p^{\ast})^{-1}$ with $1 \otimes \tr_{\varsigma_p}: \cE^{\ord,\can}_{U^p(N_1,N_2-1),\Delta,\rho} \otimes {\varsigma_p}_{\ast}\cI_{\del \cX^{\ord}_{U^p(N_1,N_2),\Delta}} \rightarrow \cE^{\ord,\can}_{U^p(N_1,N_2-1),\Delta,\rho}$ gives another $G(\bA^\infty)^{\ord,\times}$-equivariant map
	$$\tr_F: {\varsigma_p}_{\ast}\cE^{\ord,\sub}_{U^p(N_1,N_2),\Delta,\rho} \rightarrow \cE^{\ord,\sub}_{U^p(N_1,N_2-1),\Delta,\rho}$$
satisfying $\tr_F \circ \varsigma_p^{\ast} = p^{n^2[F^+:\bQ]}$ and compatible with the analogous map defined on $\{\cE^{\ord,\can}_{U,\Delta,\rho}\}_{U}$.

Denote the pushforward by $\cE^{\ord,\sub}_{U^p(N_1,N_2),\rho} = \pi^{\ord}_{\tor/\min\ast} \cE^{\ord,\sub}_{U^p(N_1,N_2),\Delta,\rho}.$ These coherent sheaves defined on $\cX^{\ord,\min}_{U^p(N_1,N_2)}$ are independent of the choice of $\Delta$ by Proposition 1.4.3.1 and  Lemma 8.3.5.2 in \cite{La4}. Note that
	$$\cE^{\ord,\sub}_{U^p(N_1,N_2),\rho} \otimes \omega_{U^p(N_1,N_2)} \cong \cE^{\ord,\sub}_{U^p(N_1,N_2),\rho \otimes (\wedge^{n[F:\bQ]} \Std^\vee)},$$
and by Lemma 5.5 in \cite{HLTT}, the pullback of $\cE^{\sub}_{U^p(N_1,N_2),\rho}$ to $\cX^{\ord,\min}_{U^p(N_1,N_2),\rho}$ is $\cE^{\ord,\sub}_{U^p(N_1,N_2),\rho}.$

Abusing notation, denote the pullback of $\cE^{\ord,\sub}_{U^p(N_1,N_2),\rho}$ to $\fX^{\ord,\min}_{U^p(N_1)}$ by $\cE^{\ord,\sub}_{U^p(N_1),\rho}$. It is independent of $N_2$, and thus, $\tr_F$ induces a $G(\bA^\infty)^{\ord,\times}$-equivariant map
	$$\tr_F: {\varsigma_p}_{\ast}\cE^{\ord,\sub}_{U^p(N_1),\rho} \rightarrow \cE^{\ord,\sub}_{U^p(N_1),\rho}$$
over $\fX^{\ord,\min}_{U^p(N_1)}$, and also induces an endomorphism on global sections.
	
\section{Classical and $p$-adic automorphic forms}

Before we define cuspidal automorphic representations on $G(\bA^\infty)$, $L_{(n)}(\bA)$ and $\GL_m(\bA_F)$, we first recall some facts about highest weights of algebraic representations of $L_{(n)}$ and $G$.

\subsection{Weights}


For each integer $0 \leq i \leq n$, let $\Lambda_{(i)}$ denote the elements of $\Lambda$ for which the last $2n - i$ coordinates are zero, and let $B_n$ denote the Borel of $G$ preserving the chain $\Lambda_{(n)}\supset \Lambda_{(n-1)} \supset ... \supset \Lambda_{(0)}.$ Let $T_n$ denote the subgroup of diagonal matrices of $G$. 


Let $X^{\ast}(T_{n/\Omega}) := \Hom(T_n \times \Spec \Omega, \bG_m \times \Spec \Omega)$ and denote by $\Phi_n \subset X^{\ast}(T_{n/\Omega})$ the set of roots of $T_n$ on $\Lie G$. The subset of positive roots with respect to $B_n$ will be denoted $\Phi^+_n$ and $\Delta_n$ will denote the set of simple positive roots. For any ring $R\subset \bR$, let $X^\ast(T_{n/\Omega})^+_R$ will denote the subset of $R$-roots $X^\ast(T_n/\Omega) \otimes_{\bZ} R$ which pair non-negatively with the simple coroots $\check{\alpha} \in X_{\ast}(T_n/\Omega) = \Hom(\bG_m \times \Spec \Omega, T_n \times \Spec \Omega)$ corresponding to the elements of $\alpha \in \Delta_n$. 

Let $\Phi_{(n)} \subset \Phi_n$ denote the set of roots of $T_n$ on $\Lie L_{(n)}$ and set $\Phi_{(n)}^+ = \Phi_{(n)} \cap \Phi_n^+$ as well as $\Delta_{(n)} = \Delta_n \cap \Phi_{(n)}$. If $R \subset \bR$ is a ring, then $X^{\ast}(T_{n/\Omega})^+_{(n),R}$ will denote the subset of $X^{\ast}(T_{n/\Omega})_{(n)} \otimes_{\bZ} R$ consisting of elements which pair non-negatively with the simple coroot $\check{\alpha} \in X_{\ast}(T_{n/\Omega})_{(n)}$ corresponding to each $\alpha \in \Delta_{(n)}$.

Recall that $L_{(n)} \times \Spec \Omega \cong \GL_1 \times \GL_n^{\Hom(F,\Omega)}$, which induces an identification 
	$$T_n \times \Spec \Omega \cong \GL_1 \times (\GL_1^n)^{\Hom(F,\Omega)},$$ and hence $X^{\ast}(T_{n/\Omega}) \cong \bZ \bigoplus (\bZ^n)^{\Hom(F,\Omega)}.$ Under this isomorphism, the image of $X^{\ast}(T_{n/\Omega})^+_{(n)}$ is the set
	$$\{(b_0,(b_{\tau,i})) \in \bZ \oplus (\bZ^n)^{\Hom(F,\Omega)} : b_{\tau,1} \geq b_{\tau,2} \geq \hdots \geq b_{\tau,n} \quad \forall \tau\}$$
Furthermore, $X^{\ast}(T_{n/\Omega})^+$ is identified with
	$$\{(b_0,(b_{\tau,i})) \in \bZ \oplus (\bZ^n)^{\Hom(F,\Omega)} : b_{\tau,1} \geq b_{\tau,2} \geq \hdots \geq b_{\tau,n} \mbox{ and } b_{\tau,1} + b_{\tau c, 1} \leq 0 \quad \forall \tau\}$$
	
Denote by $\Std$ the representation of $L_{(n)}$ on $\Lambda/\Lambda_{(n)}$ over $\bZ$. Note that the representation $\wedge^{n[F:\bQ]} \Std^{\vee}$ is irreducible with highest weight $(0,(-1,\hdots, -1)_{\tau})$. If $\rho$ is an algebraic representation of $L_{(n)}$ over $\overline{\bQ}_p$, then its highest weight lies in $X^{\ast}(T_{n/\overline{\bQ}_p})_{(n)}^+$ and uniquely up to isomorphism identifies $\rho$. Thus, for any $\underline{b} \in X^{\ast}(T_{n/\overline{\bQ}_p})_{(n)}^+$, let $\rho_{\underline{b}}$ denote the $L_{(n)}$-representation over $\overline{\bQ}_p$ with highest weight $\underline{b}$. 

Define the set of \emph{classical} highest weights $X^{\ast}(T_{n/\overline{\bQ}_p})_{\cl}^+$ as any $\underline{b} = (b_0,(b_{\tau,i})_{\tau\in\Hom(F,\overline{\bQ}_p)}) \in X^{\ast}(T_{n/\overline{\bQ}_p})^+_{(n)}$ such that $b_{\tau,1} + b_{\tau c, 1} \leq -2n$. 

We next turn to local components of automorphic representations, i.e. smooth representations of $G(\bQ_\ell)$ when $\ell \neq p$. We relate them to smooth representations of $\GL_{2n}(\bQ_\ell)$ via local base change defined below. 

\subsection{Local Base Change}
For a rational prime $\ell \neq p$, denote the primes of $F^+$ above $\bQ$ as $u_1, \cdots, u_r, v_1 \cdots v_s$, where each $u_i = w_i{}^cw_i$ splits in $F$ and none of the $v_j$ split in $F$. Note that
	$$G(\bQ_q) \cong \prod_{i=1}^r \GL_{2n}(F_{w_i}) \times H \quad \mbox{where} \quad H = \left\{ (\mu,g_i) \in \bQ_\ell^\times \times \prod_{i=1}^s \GL_{2n}(F_{v_i}) : {}^tg_i J_n {}^c g_i = \mu J_n \quad \forall i\right\}.$$ 
Here, $H$ contains a product $\prod_{i=1}^s G^1(F_{v_i}^+),$ where $G^1$ denotes the group scheme over $\cO_{F^+}$ defined by
	$$G^1(R) = \{g \in \Aut_{\cO_F \otimes_{\cO_{F^+}} R}(\Lambda \otimes_{\cO_{F^+}} R) : {}^tg J_n {}^cg = J_n\}.$$
Note that $\ker \nu \cong \RS^{\cO_{F^+}}_{\bZ} G^1.$ If $\Pi$ is an irreducible smooth representation of $G(\bQ_\ell)$ then
	$$\Pi = \left(\otimes_{i=1}^r \Pi_{w_i}\right) \otimes \Pi_H.$$
Define $\BC(\Pi)_{w_i} := \Pi_{w_i}$ and $\BC(\Pi)_{cw_i} := \Pi_{w_i}^{c,\vee}$. This does not depend on the choice of $w_i$. We call $\Pi$ \emph{unramified} at $v_i$ if $v_i$ is unramified over $F^+$ and 
	$$\Pi^{G^1(\cO_{F^+,v_i})} \neq (0).$$

Let $B^1$ denote the Borel subgroup of $G^1$ consisting of upper triangular matrices and $T^1$ the torus subgroup consisting of diagonal matrices. 

If $\Pi$ is unramified at $v_i$ then there is a character $\chi$ of $T^1(F_{v_i}^+)/T^1(\cO_{F^+,v_i})$ such that $\left.\Pi\right|_{G^1(F_{v_i}^+)}$ and $n-\Ind_{B^1(F_{v_i}^+)}^{G^1(F_{v_i}^+)} \chi$ share an irreducible subquotient with a $G^1(\cO_{F^+,v_i})$-fixed vector. Define a map between the torus of diagonal matrices of $\GL_{2n}(F_{v_i})$ and $G^1(\cO_{F^+,v_i})$:
\begin{eqnarray}
\N: T_{\GL_{2n}}(F_{v_i}) &\rightarrow& T^1(F_{v_i}^+) \\
\left(\begin{array}{ccc}t_1 & 0 & 0 \\0 & \ddots & 0 \\0 & 0 & t_{2n}\end{array}\right) & \mapsto & \left(\begin{array}{ccc}t_1/{}^c t_{2n} & 0 & 0 \\0 & \ddots & 0 \\0 & 0 & t_{2n}/{}^ct_1\end{array}\right) 
\end{eqnarray}
We define $\BC(\Pi)_{v_i}$ to be the unique subquotient of
	$$n-\Ind_{B_{\GL_{2n}}(F_{v_i})}^{\GL_{2n}(F_{v_i})} \chi \circ N$$
with a $\GL_{2n}(\cO_{F,v_i})$-fixed vector, where $B_{\GL_{2n}}(F_{v_i})$ denote the Borel subgroup of upper triangular matrices. 

\begin{lemma}[Lemma 1.1 in \cite{HLTT}]
Suppose that $\psi \otimes \pi$ is an irreducible smooth representation of
	$$L_{(n)}(\bQ_q) \cong L_{(n),\herm}(\bQ_q) \times L_{(n),\lin}(\bQ_q) = \bQ_q^\times \times \GL_n(F_q).$$
\begin{enumerate}
\item If $v$ is unramified over $F^+$ and $\pi_v$ is unramified then $n-\Ind_{P_{(n)}(\bQ_q)}^{G(\bQ_q)}(\psi \otimes \pi)$ has a subquotient $\Pi$ which is unramified at $v$. Moreover $\BC(\Pi)_v$ is the unramified irreducible subquotient of $n-\Ind_{B_{\GL_{2n}}(F_v)}^{\GL_{2n}(F_v)} (\pi_v^{c,\vee} \otimes \pi_v)$.
\item If $v$ is split over $F^+$ and $\Pi$ is an irreducible subquotient of the normalized induction $n-\Ind_{P_{(n)}(\bQ_q)}^{G(\bQ_q)}(\psi \otimes \pi)$ then $\BC(\Pi)_v$ is an irreducible subquotient of $n-\Ind_{B_{\GL_{2n}}(F_v)}^{\GL_{2n}(F_v)} (\pi_{{}^cv}^{c,\vee} \otimes \pi_v)$.
\end{enumerate}
Note that in both cases $\BC(\Pi_v)$ does not depend on $v$.
\end{lemma}

\subsection{Cuspidal automorphic representations} 

Here, we define automorphic representations on $G(\bA)$ whose finite parts will be realized in the space of global sections of $\cE^{\sub}_{U,\rho}$ on $X^{\min}_{U,\rho}$. We first recall a few definitions. Let $U(n) \subset \GL_n(\bC)$ denote the subgroup of matrices $g$ satisfying ${}^tg{}^cg = 1_n$. Define 
	$$\cK_{n,\infty} = (U(n) \times U(n))^{\Hom(F^+,\bR)} \rtimes S_2 $$
where $S_2$ acts by permuting $U(n) \times U(n)$. We can embed $\cK_{n,\infty}$ in $G(\bR) \subset \R^\times \times \prod_{\tau \in \Hom(F^+,\R)} GL_{2n}(F \otimes_{F^+,\tau} \R)$ via the map sending
	$$(g_{\tau},h_{\tau})_{\tau \in\Hom(F^+,\bR)} \mapsto  \left( 1,  \mat{(g_\tau+h_\tau)/2}{(g_\tau-h_\tau)\Psi_n/2i}{\Psi_n(g_\tau-h_\tau)/2i}{\Psi_n(g_\tau+h_\tau)\Psi_n/2}_{\tau \in \Hom(F^+,\bR)} \right),$$
and sending the nontrivial element of $S_2$ to $\left(-1,\left(\begin{smallmatrix} -1_n & 0 \\ 0 & 1_n \end{smallmatrix}\right)_{\tau \in \Hom(F^+,\bR)}\right).$
This forces $\cK_{n,\infty}$ to be a maximal compact subgroup of $G(\bR)$ such that $\cK_{n,\infty} \cap P_{(n)}(\bR)$ is a maximal compact of $L_{(n)}(\bR)$. Let $\fg = (\Lie G_{(n)}(\bR))_{\bC}$, and denote by $A_n$ the image of $\bG_m$ in $G$ via the embedding $t \mapsto t \cdot 1_{2n}$. We define a {\em cuspidal automorphic representation of $G(\bA)$} to be the twist by a character of an irreducible admissible $G(\bA^\infty) \times (\fg, \cK_{n,\infty})$-submodule of the space of cuspidal automorphic forms on the double coset space $G(\bQ)\bs G(\bA)/A_n(\bR)^0$. Furthermore, a {\em square-integrable automorphic representation of $G(\bA)$} is the twist by a character of an irreducible admissible $G(\bA^\infty) \times (\fg,\cK_{n,\infty})$-module that occurs discretely in the space of square integrable automorphic forms on $G(\bA)$.

Now let $\fl = (\Lie L_{(n)}(\bR))_{\bC}$ and let $A_{(n)}$ denote the maximal split torus in the center of $L_{(n)}$. A {\em cuspidal automorphic representation of $L_{(n)}(\bA)$} is the twist by a character of an irreducible admissible $L_{(n)}(\bA^\infty) \times (\fl, \cK_{n,\infty} \cap L_{(n)}(\bR))$-submodule of the space of cuspidal automorphic forms of $L_{(n)}(\bA)$ on the double coset space $L_{(n)}(\bQ)\bs L_{(n)}(\bA)/A_{(n)}(\bR)^0$. Define a {\em square-integrable automorphic representation of $L_{(n)}(\bA)$} as the twist by a character of an irreducible admissible $L_{(n)}(\bA^\infty) \times (\fl,\cK_{n,\infty} \cap L_{(n)}(\bR))$-submodule that occurs discretely in the space of square integrable automorphic forms on $L_n(\bA)$.

For a number field $K$ and any positive integer $m$, let $\cK_{K,\infty}$ denote a maximal compact subgroup of $\GL_m(K_{\infty})$ and let $\fg\fl = (\Lie \GL_m(K_{\infty}))_{\bC}$. Define a cuspidal automorphic representation of $GL_m(\bA_K)$ as an irreducible admissible $\GL_m(\bA^\infty_F) \times (\fg\fl, \cK_{K,\infty})$-submodule of the space of cuspidal automorphic forms on the double coset space $\GL_m(K)\bs \GL_m(\bA_K)/\bR^\times_{>0}$. Finally, by a square-integrable automorphic representation of $\GL_m(\bA_K)$, we shall mean the twist by a character of an irreducible admissible $\GL_m(\bA_K^\infty) \times (\fg\fl,\cK_{K,\infty})$-module that occurs discretely in the space of square integrable automorphic forms on $\GL_m(\bA_K)$.

We will now relate the finite parts of these automorphic representations to the global sections of the automorphic bundles defined previously.

\subsection{Global sections of automorphic bundles over the Shimura variety}
Let $\rho$ be a representation of $L_{(n)}$ on a finite $\bQ$-vector space. Define the admissible $G(\bA^\infty)$-module 
	$$H^0(X^{\min},\cE^{\sub}_{\rho}) = \lim_{\stackrel{\rightarrow}{U}} H^0(X^{\min}_U,\cE^{\sub}_{U,\rho}).$$
Note that for any neat open compact $U$, $H^0(X^{\min},\cE^{\sub}_{\rho})^{U} = H^0(X^{\min}_U,\cE^{\sub}_{U,\rho})$.
	
\begin{proposition}[Corollary 5.12 in \cite{HLTT}]\label{classical} Suppose that
$\underline{b} \in X^{\ast}(T_{n/\overline{\bQ}_p})_{\cl}^+$, and $\rho_{\underline{b}}$ is the irreducible representation of $L_{(n)}$ with highest weight $\underline{b}$. Then $H^0(X^{\min}, \cE^{\sub}_{\rho_{\underline{b}}})$ is a semisimple $G(\bA^\infty)$ module, and if $\Pi$ is an irreducible subquotient of $H^0(X^{\min},\cE^{\sub}_{\rho_{\underline{b}}})$, then there is a continuous representation
	$$R_p(\Pi): G_F \rightarrow GL_{2n}(\overline{\bQ}_p)$$
which is de Rham above $p$ and has the following  property: Suppose that $v \nmid p$ is a prime of $F$ which is
\begin{itemize}
 \item either split over $F^+$,
 \item or inert but unramified over $F^+$ and $\Pi$ is unramified at $v$;
\end{itemize}
then $$\left.\WD(R_p(\Pi)\right|_{G_{F_v}})^{\Frob-ss} \cong \rec_{F_v}(\BC(\Pi_\ell)_v|\det|_v^{(1-2n)/2}),$$
where $q$ is the rational prime below $v$. 
\end{proposition}
\begin{proof} 
Each irreducible subquotient $\Pi$ of $H^0(X^{\min},\cE^{\sub}_{\rho_{\underline{b}}})$ is the finite part of a cohomological cuspidal $G(\bA)$-automorphic representation $\pi$ by Lemma 5.11 in \cite{HLTT}, and furthermore, $H^0(X^{\min},\cE^{\sub}_{\rho_{\underline{b}}})$ is a semisimple $G(\bA^\infty)$-module. For such $\pi$, by \cite{Sh2,MW}, there is a decomposition into positive integers
	$$2n = m_1n_1 + \hdots + m_rn_r$$
and cuspidal conjugate self-dual automorphic representations $\tilde{\pi}_i$ of $\GL_{m_i}(\bA_F)$ such that for each $i \in [1,r]$,  $\tilde{\pi}_i ||\det||^{(m_i + n_i - 1)/2}$ is cohomological and satisfies the following at all primes $v$ of $F$ which are split over $F^+$
	$$\pi_v = \boxplus_{i=1}^r \boxplus_{j=0}^{n_i - 1} \tilde{\pi}_{i,v} |\det|_v^{(n_i-1)/2-j}.$$ 
These $\tilde{\pi}_i$ are automorphic representations which have Galois representations associated to them which satisfy full local-global compatibility by \cite{CH,Shin,Ca,BLGGT2,BLGHT} as described in \cite{BLGGT2}.
\end{proof}

We will refer to irreducible subquotients of $H^0(X^{\min},\cE^{\sub}_{\rho_{\overline{b}}})$ as classical cuspidal $G$-automorphic forms of weight $\rho_{\underline{b}}$.

\subsection{$p$-adic (cuspidal) $G$-automorphic forms}

Now let $\rho$ be a representation of $L_{(n)}$ on a finite locally free $\bZ_{(p)}$-module. Let $H^0(\fX^{\ord,\min},\cE^{\ord,\sub}_{\rho})$ denote the smooth $G(\bA^{\infty})^{\ord}$-module defined as
	$$H^0(\fX^{\ord,\min},\cE^{\ord,\sub}_{\rho}) := \displaystyle \lim_{\stackrel{\rightarrow}{U^p,N_1}} H^0(\fX^{\ord,\min}_{U^p(N_1)}, \cE^{\ord,\sub}_{U^p(N_1),\rho}).$$ 
For each positive integer $r$, define 
	$$H^0(\cX^{\ord,\min},\cE^{\ord,\sub}_{\rho} \otimes \bZ/p^r\bZ) := \displaystyle\lim_{\stackrel{\rightarrow}{U^p(N_1,N_2)}} H^0(\cX^{\ord,\min}_{U^p(N_1,N_2)},\cE^{\ord,\sub}_{U^p(N_1,N_2),\rho} \otimes \bZ/p^r\bZ).$$ It is a smooth $G(\bA^\infty)^{\ord}$-module of $p$-adic cuspidal $G$-automorphic forms of weight $\rho$, with the property that
	$$H^0(\cX^{\ord,\min},\cE^{\ord,\sub}_{\rho} \otimes \bZ/p^r\bZ)^{U^p(N_1,N_2)} = H^0(\cX^{\ord,\min}_{U^p(N_1,N_2)},\cE^{\ord,\sub}_{U^p(N_1,N_2),\rho} \otimes \bZ/p^r\bZ).$$
Note that mod $p^M$, there is a $G(\bA^\infty)^{\ord}$-equivariant embedding
	$$H^0(\fX^{\ord,\min},\cE^{\ord,\sub}_{\rho}) \otimes_{\bZ_p} \bZ/p^M \bZ \hookrightarrow H^0(\cX^{\ord,\min}, \cE^{\ord,\sub}_{\rho} \otimes \bZ/p^M \bZ).$$
	
Fix a neat open compact subgroup $U^p \subset G(\bA^{p,\infty})$ and integers $N_2 \geq N_1 \geq 0$, and recall that there is a canonical section $\Hasse_U \in H^0(\overline{X}^{\min}_U, \omega_U^{\otimes(p-1)})$ which is $G(\bA^{p,\infty} \times \bZ_p)$-invariant. Let $\widetilde{\Hasse}_U$ denote the noncanonical lift of $\Hasse_U$ over an open subset of $\cX^{\min}_U$. For each positive integer $M$, $\widetilde{\Hasse}_U^{p^{M-1}} \mod p^M$ glue to give a canonical $G(\bA^{\infty,p} \times \bZ_p)$-invariant section $\Hasse_{U,M}$ of $\omega^{\otimes(p-1)p^{M-1}}$ over $\cX^{\min}_U \times \Spec \bZ/p^M\bZ$.

Fix $\rho$ a representation of $L_{(n)}$ on a finite free $\bZ_{(p)}$-module then for each integer $i$, define the $G(\bA^\infty)^{\ord,\times}$-equivariant map,
\begin{eqnarray*}
H^0(\cX^{\min}_{U^p(N_1,N_2)}, \cE^{\sub}_{\rho} \otimes \omega_U^{ip^{M-1}(p-1)}) &\rightarrow& H^0(\cX^{\ord,\min}_{U^p(N_1,N_2)}, \cE^{\ord,\sub}_{\rho} \otimes \bZ/p^M\bZ) \\ f &\mapsto &(\left.f\right|_{\cX^{\ord,\min}_{U^p(N_1,N_2)}})/\Hasse^i_{U^p(N_1,N_2),M}
\end{eqnarray*}

Using the map defined above, the authors of \cite{HLTT} prove the following density theorem relating $p$-adic and classical cuspidal automorphic forms. 
	
\begin{lemma}[Lemma 6.1 in \cite{HLTT}]\label{density} Let $\rho$ be an irreducible representation of $L_{(n)}$ over $\bZ_{(p)}$. The induced map:
	$$\bigoplus_{j = r}^\infty H^0(\cX^{\min}_{U^p(N_1,N_2)},\cE^{\sub}_{U^p(N_1,N_2),\rho \otimes (\wedge^{n[F:\bQ]} \Std^{\vee})^{jp^{M-1}(p-1)}}) \rightarrow H^0(\cX^{\ord,\min}_{U^p(N_1,N_2),\rho}, \cE^{\sub,\ord}_{U^p(N_1,N_2),\rho}\otimes \bZ/p^M\bZ)$$
is surjective for any integer $r$.
\end{lemma}

\begin{proof} This lemma follows from Lemma \ref{affine} and the vanishing of $H^i(\cX^{\min}_{U^p(N_1,N_2)}, \cE^{\sub}_{U^p(N_1,N_2),\rho} \otimes \omega_{U^p(N_1,N_2)}^{\otimes k})$ as well as $H^1(\cX^{\min}_U, \bigoplus_{j=0}^\infty \cE^{\sub}_{U,\rho} \otimes \omega_{U^p(N_1,N_2)}^{\otimes p^{M-1}(p-1)j} \otimes \bZ/p^M\bZ)$ for all $i > 0$ and $k \geq 0$. For the full proof, see \cite{HLTT}.
\end{proof}

%
%
\section{The $U_p$-operator and the Main Theorem of \cite{HLTT}}

The map $\tr_F: {\varsigma_p}_{\ast} \cE^{\sub}_{U^p(N_1),\rho} \rightarrow \cE^{\sub}_{U^p(N_1),\rho}$ over $\fX^{\ord,\min}_{U^p(N_1)}$ induces an endomorphism $U_p = \tr_F$ in the endomorphism algebra of $H^0(\fX^{\ord,\min}_{U^p(N_1)},\cE^{\ord,\sub}_{U^p(N_1),\rho})_{\overline{\bQ}_p} := H^0(\fX^{\ord,\min}_{U^p(N_1)},\cE^{\ord,\sub}_{U^p(N_1),\rho})\otimes \overline{\bQ}_p$ satisfying $U_p \circ \varsigma_p = p^{n^2[F^+:\bQ]}$. By \S6.4 of \cite{HLTT}, $H^0(\fX^{\ord,\min}_{U^p(N_1)},\cE^{\ord,\sub}_{U^p(N_1),\rho})_{\overline{\bQ}_p}$ admits a slope decomposition for $U_p$ in the sense of \S6.2 of \cite{HLTT}. This means that for each $a \in \bQ$, there is a $U_p$-preserving decomposition
	$$H^0(\fX^{\ord,\min}_{U^p(N_1)},\cE^{\ord,\sub}_{U^p(N_1),\rho})_{\overline{\bQ}_p, \leq a} \oplus H^0(\fX^{\ord,\min}_{U^p(N_1)},\cE^{\ord,\sub}_{U^p(N_1),\rho})_{\overline{\bQ}_p, > a} = H^0(\fX^{\ord,\min}_{U^p(N_1)},\cE^{\ord,\sub}_{U^p(N_1),\rho})_{\overline{\bQ}_p}$$
such that $H^0(\fX^{\ord,\min}_{U^p(N_1)},\cE^{\ord,\sub}_{U^p(N_1),\rho})_{\overline{\bQ}_p, \leq a}$ is finite dimensional and satisfies:
\begin{enumerate}
\item If the roots of $f(X) \in \overline{\bQ}_p[X]$ have $p$-adic valuation $\leq a$, then the endomorphism $f(U_p)$ restricts to an automorphism of $H^0(\fX^{\ord,\min}_{U^p(N_1)},\cE^{\ord,\sub}_{U^p(N_1),\rho})_{\overline{\bQ}_p, > a}$;
\item There is a nonzero polynomial $f(X) \in \overline{\bQ}_p[X]$ with slopes $\leq a$ (i.e. $f(x) \neq 0$ and every root of $f(x)$ has $p$-adic valuation at most equal to $a$) such that the endomorphism $f(U_p)$ restricts to $0$ on $H^0(\fX^{\ord,\min}_{U^p(N_1)},\cE^{\ord,\sub}_{U^p(N_1),\rho})_{\overline{\bQ}_p, \leq a}$
\end{enumerate}

Additionally, $H^0(\fX^{\ord,\min}_{U^p(N_1)},\cE^{\ord,\sub}_{U^p(N_1),\rho})_{\overline{\bQ}_p, \leq a}$ is $G(\bA^\infty)^{\ord,\times}$-invariant.
	
\begin{theorem} Suppose that $n > 1$ and that $\rho$ is an irreducible algebraic representation of $L_{(n),\lin}$ on a finite-dimensional $\overline{\bQ}_p$-vector space. Suppose also that $\pi$ is a cuspidal automorphic representation of $L_{(n),\lin}(\bA)$ so that $\pi_\infty$ has the same infinitesimal character as $\rho^{\vee}$ and that $\psi$ is a continuous character of $\bQ^\times \bs \bA^\times /\bR^\times_{>0}$ such that $\left. \psi\right|_{\bZ_p^\times} = 1$. Then for all $N \in \bZ_{>0}$ sufficiently large, there exists a representation $\rho(N,j)$ of $L_{(n)}$ over $\bZ_{(p)}$, $a(N,j) \in \bQ$, and an admissible representation $\Pi_j'$ of $H^0(\fX^{\ord,\min},\cE^{\sub}_{\rho(N,j)})_{\overline{\bQ}_p,\leq a(N,j)}$ such that every irreducible subquotient $\pi_j$ of $\Ind_{P_{(n)}^+(\bA^{p,\infty})}^{G(\bA^{p,\infty})} (\pi^\infty||\det||^N \times \psi^{\infty})$ is a subquotient of $\Pi_j'$.
\end{theorem}
\begin{proof} Combine Corollary 1.13, Corollary 6.12, Corollary 6.17, Lemma 6.20, and Corollary 6.25 in \cite{HLTT}
\end{proof}

Our next step is to consider properties of the Galois representations associated to the irreducible $G(\bA^\infty)^{\ord,\times}$-subquotients of $H^0(\fX^{\min,\ord},\cE^{\sub}_{\rho})_{\overline{\bQ}_p,\leq a}$, as constructed in Corollary 6.13 in \cite{HLTT}. In order to prove local-global compatibility at all primes $\ell$ such that $\ell \neq p$, we strengthen the construction of Galois representations associated to irreducible admissible $G(\bA^\infty)^{\ord,\times}$-subquotients of $H^0(\fX^{\min,\ord},\cE^{\sub}_{\rho})_{\overline{\bQ}_p}$, i.e. Galois representations associated to $p$-adic cuspidal $G$-automorphic forms of weight $\rho$ (see Proposition 6.5 of \cite{HLTT}). These Galois representations are constructed using the following two facts we have already recalled:
\begin{enumerate}
\item Proposition \ref{classical}: Classical cuspidal $G$-automorphic forms of classical weight $\rho$ have Galois representations associated to them; furthermore, they satisfy full local-global compatibility at all primes $\ell$ such that $\ell \neq p$. 
\item Lemma \ref{density}: For any integer $M$, every $p$-adic cuspidal $G$-automorphic form of any regular weight $\rho$ ``is congruent mod $p^M$ to'' some classical cuspidal $G$ automorphic form of classical weight $\rho'$ which is of the form $\rho' = \rho \otimes (\wedge^{n[F:\bQ]} \Std^{\vee})^{(p-1)p^{M-1}j}$ for some integer $j$.
\end{enumerate}
To prove local-global compatibility when $\ell \neq p$, we will use these two results to reconstruct the Galois representations associated to $p$-adic cuspidal automorphic forms on $G$ of regular weight $\rho$, but we will consider the action of a larger Hecke algebra than in \cite{HLTT} on the $p$-adic automorphic spaces $H^0(\cX^{\ord,\min}_{U^p(N_1,N_2)},\cE^{\ord,\sub}_{\rho} \otimes \bZ/p^M\bZ)$ and $H^0(\fX^{\ord,\min}_{U^p(N_1)},\cE^{\ord,\sub}_{\rho})$ as well as the classical automorphic spaces $H^0(\cX^{\min}_{U^p(N_1,N_2)},\cE^{\sub}_{\rho})$. 

\section{Hecke Algebras away from $p$}

Let $S$ denote the set of ``bad'' rational primes consisting of $p$ and the primes where $F$ ramifies. Inside $S$, let $S_{\spl}$ denote the set of rational primes $\ell \in S \bs \{p\}$ such that there is at least one prime $v$ of $F$ above $\ell$ which splits over $F^+$. Finally, let $S^p = S \bs \{p\}$.

We will denote by $\cS^{\spl}$ the set of primes $v$ of $F$ such that $v \mid \ell \in S \bs\{p\}$ and $v$ is \emph{not} split over $F^+$. Finally, for each conjugate pair of primes $\{v,{}^cv\} \not\subset \cS^{\spl}$ which divide a rational prime $\ell \in S$, choose exactly one of $\{v,{}^cv\}$ to put into a set $\cS_{\spl}$ and the other in $\overline{\cS}_{\spl}$. Note that the disjoint union $\cS^{\spl} \sqcup \cS_{\spl} \sqcup \overline{\cS}_{\spl}$ contains all primes of $F$ above $S \bs \{p\}$. Finally, let $\cS^{\spl,+}$ denote the primes of $F^+$ below $\cS^{\spl}$.

For $\ell \in S_{\spl}$, identify $G(\bQ_\ell) \cong \prod_{v \in \cS_{\spl}} \GL_{2n}(F_v) \times H,$ where $$H = \left\{ (\mu,g_i) \in \bQ_\ell^\times \times \prod_{\cS^{\spl,+} \ni w \mid \ell} \GL_{2n}(F_{w}) : {}^tg_iJ_n{}^cg_i = \mu J_n \quad \forall i\right\}.$$

\subsection{At unramified primes}

We recall the definition of the unramified Hecke algebra as in \S6 in \cite{HLTT}.
Fix a neat open compact subgroup $U^p = G(\widehat{\bZ}^S) \times U_{S^p} \subset G(\bA^{p,\infty})$. Suppose that $v$ is a place of $F$ above a rational prime $\ell \notin S$ and let $i \in \bZ$.

By Satake, there is an element $T_v^{(i)} \in \bQ[G(\bZ_\ell)\bs G(\bQ_\ell)/G(\bZ_\ell)]$ such that if $\Pi_\ell$ is an unramified representation of $G(\bQ_\ell)$ and then its eigenvalue on $\Pi_\ell^{G(\bZ_\ell)}$ is equal to
	$$\tr \rec_{F_v}(\BC(\Pi_\ell)_v)|\det|_v^{(1-2n)/2}(\Frob_v^i).$$
If $v$ is an unramified prime of $F$ which splits over $F^+$, then we can write the Hecke operator $T_v^{(1)}$ as the double coset
	$$G(\bZ_\ell)\left(\begin{array}{cccc}1 & 0 & \cdots & 0 \\0 & \ddots & 0 & \vdots \\ \vdots & 0  & 1 & 0 \\0  & \hdots & 0 & \varpi_v\end{array}\right)G(\bZ_\ell)$$
where $\varpi_v$ denotes a uniformizer of $F_v$.

For each unramified prime $v$ of $F$ and each integer $i \in \bZ$, there exists an integer $d_v^{(i)} \in \bZ$ such that 
	$$d_{v}^{(i)} T_v^{(i)} \in \bZ[G(\bZ_\ell)\bs G(\bQ_\ell))\bs G(\bZ_\ell)].$$

Let $\cH^S_{\bZ_p} := \bZ_p[G(\widehat{\bZ}^S)\bs G(\bA^S)/G(\widehat(\bZ)^S)]$ denote the abstract unramified Hecke algebra. For any two integers $N_2 \geq N_1 \geq 0$ and let $\rho$ be a representation of $L_{(n)}$ over $\bZ_{(p)}$. The Hecke algebra $\cH^S_{\bZ_p}$ has an action on the classical and $p$-adic spaces $H^0(\cX^{\min}_{U^p(N_1,N_2)},\cE^{\sub}_{U^p(N_1,N_2),\rho})$, $H^0(\fX^{\ord,\min}_{U^p(N_1)},\cE^{\ord,\sub}_{U^p(N_1),\rho})$, and $H^0(\cX^{\ord,\min}_{U^p(N_1,N_2)},\cE^{\ord,\sub}_{U^p(N_1,N_2),\rho} \otimes \bZ/p^M\bZ)$ induced from the action of $G(\bA^S)$.  Denote by $\bT^S_{U^p(N_1,N_2),\rho}$ the image of $\cH^S_{\bZ_p}$ in the endomorphism algebra $\End_{\bZ_p}(H^0(\cX^{\min}_{U^p(N_1,N_2)},\cE^{\sub}_{U^p(N_1,N_2),\rho}))$. Furthermore, if $W \subset H^0(\fX^{\ord,\min}_{U^p(N_1)},\cE^{\ord,\sub}_{U^p(N_1),\rho})$ (resp. $W \subset H^0(\cX^{\ord,\min}_{U^p(N_1,N_2)},\cE^{\ord,\sub}_{U^p(N_1,N_2),\rho} \otimes \bZ/p^M\bZ)$) is a finitely-generated $\bZ_p$-submodule invariant under the action of the algebra $\cH^S_{\bZ_p}$, then let $\bT^{\ord,S}_{U^p(N_1,N_2),\rho}(W)$ (resp. $\bT^{\ord,S}_{U^p(N_1,N_2),\rho,M}(W)$) denote the image of $\cH^S_{\bZ_p}$ in $\End_{\bZ_p}(W)$. 

For each $v$, let $\tilde{T_v}^{(i)}$ denote the image of $d_{v}^{(i)} T_v^{(i)}$ in any $\cH^S_{\bZ_p}$-algebra $\bT$ via the canonical map $\cH^S_{\bZ_p} \rightarrow \bT$.

\subsection{At ramified primes which are split over $F^+$}\label{ramified}

Suppose that $v \in \cS_{\spl} \sqcup \overline{\cS}_{\spl}$ is a place of $F$ above a rational prime $\ell$, and let $\sigma_v$ denote an element of $W_{F_v}$, the Weil group of $F_v$. Let $\cB$ denote a fixed Bernstein component; it is a subcategory of the smooth representations of $\GL_{2n}(F_v)$. Every component $\cB$ is uniquely associated to an inertial equivalence class $(M,\omega)$, where $M$ denotes a Levi subgroup of $\GL_{2n}(F_v)$ and $\omega$ is a supercuspidal representation of $M$. (Recall that two inertial classes $(M,\omega)$ and $(M',\omega')$ are equivalent if there exists $g \in G$ and an unramified character $\chi$ of $M'$ such that $M = g^{-1}Mg$ and $\omega' = \chi \otimes \omega(g\cdot g^{-1}).$) Then, $\cB$ is defined to be the full subcategory of smooth representations of $\GL_{2n}(F_v)$ consisting of those representations all of whose irreducible subquotients have {\em inertial support} equivalent to $(M,\omega)$. This implies that there exists some $(M',\omega') \sim (M,\omega)$ such that $\pi$ occurs as a composition factor of the parabolic induction $\Ind_{P_M}^{\GL_{2n}(F_v)}(\omega')$ where $\omega'$ is an irreducible supercuspidal representation  and $P_M$ is a parabolic subgroup of $\GL_{2n}(F_v)$ with Levi $M$. 

Let $\fz_{\cB} = \fz_{[M,\omega]}$ denote the Bernstein center of $\cB$, which is the image under the idempotent $e_{\cB}$ associated to $\cB$ of $$\lim_{\stackrel{\longleftarrow}{ K}} \cZ(\bC[K\bs\GL_{2n}(F_v)/K]),$$ the inverse limit over open compact subgroup $K$ of the centers of the complex Hecke algebra for $\GL_{2n}(F_v).$ 

\begin{proposition}[Proposition 3.11 in \cite{C}] For an inertial equivalence class $[M,\omega]$, there is a representative $(M,\omega)$ which can be defined over $\overline{\bQ}$. Let $E \subset \overline{\bQ}$ denote a sufficiently large finite-degree normal field over which $\omega$, $\rec(\omega)$, $\cB_{[M,\omega]}$, $\fz_{[M,\omega]}$ are all defined over $E$. Let $E[\cB_{[M,\omega]}]$ denote the affine coordinate ring of the variety associated to $\cB_{[M,\omega]}$. Then there exists a unique pseudocharacter of dimension $2n$
	$$T^{\cB} = T^{[M,\omega]}: W_{F_v} \rightarrow E[\cB] = \fz_{\cB}$$
such that for all irreducible smooth representations $\pi$ of $\cB$ and $\sigma_v \in W_{F_v}$,
	$$T^{\cB}(\sigma_v)(\pi) = \tr \rec_{F_v}(\pi)(\sigma_v).$$
\end{proposition}


For a Bernstein component $\cB$ and $\sigma \in W_{F_v}$, let $T_{v,\cB,\sigma}$ denote the twist of $T^{\cB}(\sigma)$ such that $T_{v,\cB,\sigma}(\pi) = \tr \rec_{F_v}(\pi|\det|_v^{(1-2n)/2})(\sigma)$ if $\pi$ is a smooth irreducible representation in $\cB$. Multiplying $T_{v,\cB,\sigma}$ by $e_{\cB}$ if necessary, we may suppose that $T_{v,\cB,\sigma}$ acts as $0$ on all irreducible $\pi \notin \cB$. 

For each prime $v \in \cS_{\spl}$, let $v \mid \ell$ be the rational prime below $v$. Let $\cB_v = \cB$ be a Bernstein component, and let $e_{\cB}$ denote the projector element such that for any smooth irreducible representation $\pi$ of $\GL_{2n}(F_v)$, $e_{\cB}(\pi) = \pi$ if and only if $\pi \in \cB$. From \cite{BD}, there is a compact open subgroup $K$ of $\GL_{2n}(F_v)$ for which we may find a finite union of Bernstein components $\fB = \fB_v$ containing $\cB_v$ with the following property: If $\pi_v$ is an irreducible smooth representation of $\GL_{2n}(F_v)$, then $\pi_v^K$ is nonzero if and only if $\pi_v$ belongs to one of the Bernstein components in $\fB$. We will denote this compact open subgroup by $K_{\fB} = K_{\fB_v}$; note that all irreducible smooth representations inside $\cB$ have a fixed vector under $K_{\fB}$. More generally, for every $\cB' \subset \fB$, $\fz_{\cB'}$ embeds in the center of $\cH(\GL_{2n},K_{\fB})_{\bC} = \bC[K_{\fB}\backslash \GL_{2n}(F_v) /K_{\fB}]$ via multiplication by the characteristic function of $K_{\fB}$. Let $\fz_{\fB_v} = \fz_{\fB} := \im(\prod_{\cB' \subset \fB} \fz_{\cB'} \hookrightarrow \cH(\GL_{2n},K_{\fB})_{\bC})$. 

For $\ell \in S_{\spl}$, assume $K_\ell$ is an open compact subgroup of $G(\bQ_\ell)$ such that under the identification $G(\bQ_\ell) \cong \prod_{\cS_{\spl} \ni v \mid \ell} \GL_{2n}(F_v) \times H,$ we can decompose
	$$K_\ell = \prod_{\cS_{\spl} \ni v \mid \ell} K_{\fB_v} \times K_H,$$ where $K_H$ is an open compact of $H$. 
	
	 If $v \in \cS_{\spl}$ divides the rational prime $\ell$ and $\fB$ is a Bernstein component, then for any $\sigma \in W_{F_v}$ we can find an element of $\fz_{\fB}$, which we will denote by $T_{v,\fB,\sigma}$, such that its eigenvalue on the $K_{\ell}$-fixed vectors of an irreducible representation $\pi$ of $G(\bQ_\ell)$ in $\cB$ is
	$$
	\tr \rec_{F_v} (\pi_v|\det|_v^{(1-2n)/2})(\sigma).$$
(On the other hand, if $\pi_v
\notin \fB$, then $\pi^{K_\ell}$ is trivial and $T_{v,\fB,\sigma}$ acts as 0).
This element $T_{v,\fB,\sigma}$ is the image in $\fz_{\fB}$ of $\prod_{\cB' \subset \fB} T_{v,\cB',\sigma} \in \prod_{\cB'\subset \fB} \fz_{\cB'}$. It is independent of $\pi$. Furthermore, for each $\varphi \in \Aut(\bC)$, we have that ${}^{\varphi}\fB = \fB$ and additionally,
	$${}^{\varphi} \rec_{F_v}(\pi_v|\det|_v^{(1-2n)/2}) \cong \rec_{F_v}({}^\varphi\pi_v|\det|_v^{(1-2n)/2}).$$ 
Thus, we have that ${}^\varphi T_{v,\fB,\sigma} = T_{v,\fB,\sigma}$, and so $T_{v,\fB,\sigma} \in \bQ[K_\ell\backslash G(\bQ_\ell)/K_\ell].$

Define
	$$\fz_\ell^0 := \prod_{\cS_{\spl} \ni v \mid \ell} (\fz_{\fB_v} \cap \bZ[K_{\fB}\backslash \GL_{2n}(F_v) /K_{\fB}]).$$
Then $\fz_\ell^0$ lies in the center of $\bZ[K_\ell\backslash G(\bQ_\ell)/K_\ell]$. Note that for any element $T \in \fz_\fB$, there exists a nonzero integer $d(T) \in \bZ $ such that $d(T)T\in \fz_\ell^0,$ where $v \mid \ell$. Thus, we can choose $d(T_{v,\fB,\sigma}) \in \bZ^\times$ such that
	$$d(T_{v,\fB,\sigma})T_{v,\fB,\sigma} \in \bZ[K_\ell\backslash G(\bQ_\ell)/K_\ell],$$
so $d(T_{v,\fB,\sigma})T_{v,\fB,\sigma} \in \fz_\ell^0$.

For each $v \in \cS_{\spl} \sqcup \overline{\cS}_{\spl}$, fix a Bernstein component $\cB_v$ and let $\fB_v$ be the disjoint union as above. Assume $U^p$ is a neat open compact subgroup of $G(\bA^{p,\infty})$ such that 
	\begin{equation}\label{compact} U_\ell = \prod_{\cS_{\spl}\sqcup\overline{\cS}_{\spl} \ni v \mid \ell} K_{\fB_v} \times U_H
	\end{equation}
	Let $\cH_{\spl, \bZ_p} := \left(\bigotimes_{\ell \in S_{\spl}} \fz_\ell^0\right)$ be the abstract ramified Hecke algebra. For any two integers $N_2 \geq N_1 \geq 0$ and algebraic representation $\rho$ of $L_{(n)}$ over $\bZ_{(p)}$, recall that the classical space $H^0(\cX^{\min}_{U^p(N_1,N_2)},\cE^{\sub}_{U^p(N_1,N_2),\rho})$ has an action of $G(\bA^{p,\infty})$ which induces an action of $\cH_{\spl,\bZ_p}$, and similarly, the $p$-adic spaces $H^0(\fX^{\ord,\min}_{\rho},\cE^{\ord,\sub}_{\rho})$ and $H^0(\cX^{\ord,\min}_{\rho},\cE^{\ord,\sub}_{\rho} \otimes \bZ/p^M\bZ)$ have an action of $G(\bA^{p,\infty})$, which similarly induces an action of $\cH_{\spl,\bZ_p}$. Let $\bT^p_{U^p(N_1,N_2),\rho}$ denote the image of $\cH^p_{\bZ_p} := \cH^S_{\bZ_p} \otimes \cH_{\spl,\bZ_p}$ in $\End_{\bZ_p}(H^0(\cX^{\min}_{U^p(N_1,N_2)},\cE^{\sub}_{U^p(N_1,N_2),\rho}))$. 
	
	Furthermore, if $W \subset H^0(\fX^{\ord,\min}_{U^p(N_1)},\cE^{\ord,\sub}_{U^p(N_1),\rho})$ (resp. $W \subset H^0(\cX^{\ord,\min}_{U^p(N_1,N_2)},\cE^{\ord,\sub}_{U^p(N_1,N_2),\rho} \otimes \bZ/p^M\bZ)$) is a finitely generated $\bZ_p$-submodule invariant under the action of the algebra $\cH^p_{\bZ_p}$, then let $\bT^{\ord,p}_{U^p(N_1,N_2),\rho}(W)$ (resp. $\bT^{\ord,p}_{U^p(N_1,N_2),\rho,M}(W)$) denote the image of $\cH^p_{\bZ_p}$ in $\End_{\bZ_p}(W)$. 
	
For each $v \in \cS_{\spl} \sqcup \overline{\cS}_{\spl}$, let $\tilde{T}_{v,\fB,\sigma}$ denote the image of $d(T_{v,\fB,\sigma})T_{v,\fB,\sigma}$ in any $\cH^p_{\bZ_p}$-algebra $\bT$.

\section{Interpolating the Hecke action}

The main goal of this section is to prove the following proposition. 

\begin{proposition}\label{localglobalss} Let $\rho$ be an algebraic representation of $L_{(n)}$ over $\bZ_{(p)}$. Suppose that $\Pi$ is an irreducible quotient of an admissible $G(\bA^{\infty})^{\ord,\times}$-submodule $\Pi'$ of $H^0(\fX^{\ord,\min},\cE^{\ord,\sub}_{\rho}) \otimes \overline{\bQ}_p$. Then there is a continuous semisimple representation
	$$R_p(\Pi):G_F \rightarrow \GL_{2n}(\overline{\bQ}_p)$$
with the following property: If $\ell \neq p$ is a rational prime and $v \mid \ell$ is a prime of $F$ that splits over $F^+$. Then
	$$\WD(R_p(\Pi)_{G_{F_v}})^{ss} \cong \rec_{F_v}((\Pi_\ell)_v|\det|_v^{(1-2n)/2})^{ss}.$$
\end{proposition}

Proposition 6.5 in \cite{HLTT} proves the existence of $R_p(\Pi)$ and its local-global compatibility at primes $v \mid \ell \notin S$. We recall their argument in these cases and extend the local-global compatibility results to primes $v \in \cS_{\spl}$. 

Fix  $\rho$, $\Pi$, and $\Pi'$ as in the proposition. For each $v \in \cS_{\spl}$, let $\cB_v$ denote the Bernstein component containing $\BC(\Pi_\ell)_v$. Let $\fB_v$ be a disjoint union of Bernstein components containing $\cB_v$ such that there is an open compact subgroup $K_{\fB_v}$ of $\GL_{2n}(F_v)$ and an irreducible representation of $\GL_{2n}(F_v)$ with a nontrivial $K_{\fB_v}$-fixed vector is contained in $\fB_v$. Choose a neat open compact subgroup $U^p \subset G(\bA^{p,\infty})$ such that $U_\ell = \prod_{\cS_{\spl}\sqcup\overline{\cS_{\spl}}\ni v \mid \ell} K_{\fB_v} \times U_H$ for each $\ell \in S_{\spl}$ as well as an integer $N$ such that $\Pi^{U^p(N)} \neq (0)$. Recall that $\cH_{\spl,\bZ_p} := (\bigotimes_{\ell \in S_{\spl}} \fz_\ell^0)$ associated to the Bernstein components $\cB_v$ and disjoint unions $\fB_v$ and open compact subgroups $K_{\fB_v}$ fixed above for $v \in \cS_{\spl}$, and let $\cH^p_{\bZ_p} = \bZ_p[G(\widehat{\bZ}^S)\bs G(\bA^S) / G(\widehat{\bZ}^S)] \otimes_{\bZ_p} \cH_{\spl,\bZ_p}$ as before.

We first show the existence and local-global compatibility of a Galois representations associated to irreducible subquotients of the classical space $H^0(\cX^{\min}_{U^p(N_1,N_2)},\cE^{\sub}_{\rho \otimes (\wedge^{n[F:\bQ]} \Std^{\vee})^{\otimes t}})$ for $t$ sufficiently large. This follows from Proposition 5.2, Lemma 5.11 in \cite{HLTT}, and the fact that for $t$ sufficiently large, $({\rho \otimes (\wedge^{n[F:\bQ]} \Std^{\vee})^{\otimes t}) \otimes \overline{\bQ}}$ decomposes into representations of $L_{(n)}$ with classical highest weight. It will be most relevant to write this result in terms of pseudorepresentations.

\begin{lemma} For $t$ sufficiently large, there is a continuous pseudorepresentation
	$$T_t:G_F^S \rightarrow \bT^p_{U^p(N_1,N_2),\rho \otimes (\wedge^{n[F:\bQ]} \Std^{\vee})^{\otimes t}} \qquad \mbox{s.t. } \quad
	\begin{cases} d(T_{v,\fB_v,\sigma})T(\sigma) = \tilde{T}_{v,\fB_v,\sigma} & \mbox{ if } v \in \cS_{\spl} \sqcup \overline{\cS}_{\spl} \\
				d_v^{(i)}T(\Frob_v^i) = \tilde{T}_v^{(i)} & \mbox{ if } v\mid \ell \notin S.
	\end{cases}$$

\end{lemma}

\begin{proof} First, assume that $\rho \otimes \overline{\bQ}_p$ is irreducible. Let $(b_0,(b_{\tau,i})) \in X^\ast(T_{n/\overline{\bQ}_p})_{(n)}^+$ denote the highest weight of $\rho \otimes \overline{\bQ}_p$. If $t \in \bZ$ satisfies the inequality
	$$-2n \geq (b_{\tau,1} - t(p-1)) + (b_{\tau c,1} - t(p-1))$$
and $\rho_t := \rho \otimes (\wedge^{n[F:\bQ]} \Std)^{\otimes t}$, then by Proposition 5.1,
	$$\bT^p_{U^p(N_1,N_2),\rho_t} \otimes \overline{\bQ}_p \cong \bigoplus_{\Pi} \overline{\bQ}_p$$ where the sum runs over irreducible admissible representations of $G(\bA^\infty)$ with $\Pi^{U^p(N_1,N_2)} \neq (0)$ which occur in $H^0(X^{\min} \times \Spec \overline{\bQ}_p, \cE^{\sub}_{\rho_t})$. Further, from corollary 5.12 in \cite{HLTT}, we deduce that there is a continuous representation
		\begin{eqnarray}\label{heckerep}
		r_{\rho_t}: G_F^S \rightarrow \GL_{2n}(\bT^p_{U^p(N_1,N_2),\rho_t} \otimes \overline{\bQ}_p) \qquad \mbox{s.t.} \quad \begin{cases} \tr r_{\rho_t}(\Frob_v^i) = T_v^{(i)} & \mbox{ if } v \mid \ell \notin S \\
		\tr r_{\rho_t}(\sigma) = T_{v,\fB,\sigma} & \mbox{ if } v \in \cS_{\spl}	\sqcup \overline{\cS}_{\spl}	\end{cases}
		\end{eqnarray}
Let $T_t := \tr r_{\rho_t}$. Note that if $v\mid \ell \notin S$ is a prime of $F$ which is split over $F^+$, then $T_t(\Frob_v) = T_v^{(1)} \in \bT^p_{U^p(N_1,N_2),\rho_t}$, thus by Cebotarev density theorem, $T_t: G_F^S \rightarrow \bT^p_{U^p(N_1,N_2),\rho_t}$.

For general $\rho$, recall that algebraic representations of $L_{(n)}(\bZ_p)$ are semisimple, and so we can construct from the Galois representations associated to the irreducible constituents of $\rho \otimes \overline{\bQ}_p$ a continuous representation $r: G_F^S \rightarrow \GL_{2n}(\bT^p_{U^p(N_1,N_2),\rho \otimes (\wedge^{n[F:\bQ]} \Std)^{\otimes t}} \otimes \overline{\bQ}_p)$ for sufficiently large $t$ whose trace satisfies the desired properties.
\end{proof}

Combining the above lemma with Lemma 5.2, we have the following corollaries.  

\begin{corollary} If $W$ is a finitely generated $\cH^p_{\bZ_p}$-invariant submodule of either $H^0(\cX^{\ord,\min}_{U^p(N_1,N_2)}, \cE^{\ord,\sub}_{\rho} \otimes \bZ/p^M\bZ)$ or
$H^0(\fX^{\ord,\min}_{U^p(N)},\cE^{\ord,\sub}_{\rho}),$
then there is a continuous pseudorepresentation
	$$T: G_F^S \rightarrow \bT^{\ord,p}_{U^p(N_1,N_2),\rho,M}(W)  \qquad \mbox{s.t. } \quad
	\begin{cases} d(T_{v,\fB_v,\sigma})T(\sigma) = \tilde{T}_{v,\fB_v,\sigma} & \mbox{ if } v \in \cS_{\spl} \sqcup \overline{\cS}_{\spl} \\
				d_v^{(i)}T(\Frob_v^i) = \tilde{T}_v^{(i)} & \mbox{ if } v\mid \ell \notin S.
	\end{cases}$$
\end{corollary}

\begin{proof}
It suffices to show that for finitely generated $W \subset H^0(\cX^{\ord,\min}_{U^p(N_1,N_2)},\cE^{\ord,\sub}_{\rho} \otimes \bZ/p^M\bZ)$ such a pseudorepresentation exists since there is an $G(\bA^\infty)^{\ord}$-equivariant embedding
	$$H^0(\fX^{\ord,\min}_{U^p(N_1)},\cE^{\ord,\sub}_{\rho}) \otimes \bZ/p^M\bZ \hookrightarrow H^0(\cX^{\ord,\min}_{U^p(N_1,N_2)},\cE^{\ord,\sub}_{\rho} \otimes \bZ/p^M\bZ).$$ Since $W$ is finitely generated, there exists $k \in \bZ$ such that 
	$$W \subset \Im(\bigoplus_{j = r}^k H^0(\cX^{\min}_{U^p(N_1,N_2)},\cE^{\sub}_{U^p(N_1,N_2),\rho_{jp^{M-1}(p-1)}}) \rightarrow H^0(\cX^{\ord}_{U^p(N_1,N_2)}, \cE^{\sub,\ord}_{U^p(N_1,N_2),\rho} \otimes \bZ/p^M\bZ))$$
Since the above map is $G(\bA^\infty)^{\ord,\times}$-equivariant, we see that for $r$ sufficiently large, by Lemma 4.3, there is a continuous pseudorepresentation $T_r: G_F^S \rightarrow \bT^p_{U^p(N_1,N_2),\rho_{rp^{M-1}(p-1)}}$. If we take $r$ to be sufficiently large, then we can compose to get
	$$T: G_F^S \rightarrow \bigoplus_{j = r}^k \bT^p_{U^p(N_1,N_2),\rho_{jp^{M-1}(p-1)}} \rightarrow \bT^{\ord,p}_{U^p(N),\rho_{jp^{M-1}(p-1)}}(W) \quad \mbox{s.t.} \quad \begin{cases} d(T_{v,\fB_v,\sigma})T(\sigma) = \tilde{T}_{v,\fB_v,\sigma} & \mbox{ if } v \in \cS_{\spl} \\
				d_v^{(i)}T(\Frob_v^i) = \tilde{T}_v^{(i)} & \mbox{ if } v\mid q \notin S.
	\end{cases}$$
\end{proof}

We use the pseudorepresentations constructed above to finish the proof of Proposition 9.1.

\begin{proof}[Proof of Proposition \ref{localglobalss}]

Since $(\Pi')^{U^p(N)}$ is finite dimensional, it is a closed subspace of $H^0(\fX^{\ord,\min},\cE^{\ord,\sub}_{\rho}) \otimes \overline{\bQ}_p$ preserved by the action of $\cH^p_{\bZ_p}$,  we have by Corollary 9.3 that there is a continuous pseudorepresentation
	$$T: G_F^S \rightarrow \bT^{\ord,p}_{U^p(N_1,N_2),\rho}((\Pi')^{U^p(N_1,N_2)})  \qquad \mbox{s.t. } \quad
	\begin{cases} d(T_{v,\fB_v,\sigma})T(\sigma) = \tilde{T}_{v,\fB_v,\sigma} & \mbox{ if } v \in \cS_{\spl} \\
				d_v^{(i)}T(\Frob_v^i) = \tilde{T}_v^{(i)} & \mbox{ if } v\mid \ell \notin S.
	\end{cases}$$
Since there is a $\cH^p_{\bZ_p}$-equivariant map $(\Pi')^{U^p(N)} \surjects \Pi^{U^p(N)}$, there is a map $\varphi_{\Pi}:\bT^{\ord,p}_{U^p(N_1,N_2),\rho}((\Pi')^{U^p(N_1,N_2)}) \rightarrow \overline{\bQ}_p$ sending a Hecke operator to its eigenvalue on $(\Pi)^{U^p(N_1,N_2)}$. Composing $\varphi_{\Pi} \circ T$ gives a pseudorepresentation which by \cite{T} is the trace of a continuous semisimple Galois representation satisfying the semi-simplified local-global compatibility at the primes away from $\cS^{\spl}$. The proposition follows from the main theorem on pseudorepresentations (see \cite{T}).
\end{proof}

\section{Bounding the monodromy}

Let $\ell \neq p$ be distinct primes and $v$ a prime of $F$ above $\ell$ such that $v$ splits over $F^+$, i.e. $\ell \in S \bs \{p\}$ and $v \in \cS_{\spl} \sqcup \overline{\cS}_{\spl}.$ The main result of this section is as follows. 
\begin{proposition}\label{precthm} Suppose $\rho$ is an algebraic representation of $L_{(n)}$ over $\bZ_{(p)}$ and that $\Pi$ is an irreducible quotient of an admissible $G(\bA^{\infty})^{\ord,\times}$-submodule $\Pi'$ of $H^0(\fX^{\ord,\min}, \cE_{\rho}^{\ord,\sub})\otimes_{\bQ_p} \overline{\bQ}_p$. Then the continuous semi-simple representation $R_{p,\imath}(\Pi)$ satisfies for $\cS^{\spl} \sqcup \overline{\cS}^{\spl} \ni v \mid \ell \neq p $, i.e.~for all primes of $F$ which split over $F^+,$
	$$\WD(\left.R_{p,\imath}(\Pi)\right|_{W_{F_v}})^{\Frob-ss} \prec \rec_{F_v}(\BC(\Pi_\ell)_v|\det|_v^{(1-2n)/2})$$
where $\prec$ is defined below. 
\end{proposition}

Let $(\sigma,N)$ be a Weil-Deligne representation of $W_{F_v}$ over $\overline{\bQ}_p$, where $\sigma: W_{F_v} \rightarrow \GL(V)$. Let $\cW$ denote the set of equivalence classes of irreducible representations of $W_{F_v}$ over $\overline{\bQ}_p$ with open kernel, where two representations $s,s'$ of $W_{F_v}$ are in the same equivalence if $s \cong s' \otimes \chi \circ \det$ for some unramified character $\chi$. 
We can decompose any Weil-Deligne representation into isotypic components indexed by these equivalence classes of $\cW$, i.e.
	$$\sigma \cong \bigoplus_{\omega \in \cW} \sigma[\omega] \qquad V \cong \bigoplus_{\omega \in \cW} V[\omega],$$
where $\sigma[\omega]: W_{F_v} \rightarrow \GL(V[\omega])$ is a Weil representation with all irreducible subquotients lying in $\omega \in \cW$. $N$ preserves isotypic components of $\sigma$, thus it preserves $V[\omega]$. If $N[\omega]$ denotes $N$ restricted to $V[\omega]$, then $(\sigma[\omega], N[\omega])$ is a Weil-Deligne representation. Recall from \cite{Tate} that there is an indecomposable Weil-Deligne representation $(\Sp(m),N(m))$ of dimension $m$ with nilpotent matrix of degree exactly $m$.

\begin{definition}\label{prec}
 For each $\omega \in \cW$, and for each Weil-Deligne representation $(\sigma,N)$, 
there exists a unique decreasing sequence of non-negative integers $m_1(\sigma, N,\omega) \geq m_2(\sigma,N,\omega) \geq \hdots$ with an associated sequence of $s_1, s_2, \hdots \in \omega$ such that
	$$\sigma[\omega]^{\Frob-ss} \cong \bigoplus_{s_i \in \omega} s_i \otimes \Sp(m_{i,\omega}(\sigma,N)).$$
	The sequence $(m_{i,\omega}(\sigma,N))_i$ is a partition of the integer $\dim(\sigma[\omega])/\dim(s_i)$ for any $s_i \in \omega$. If $(\sigma',N')$ is another Weil-Deligne representation,
then we define
	$$(\sigma,N) \prec (\sigma',N')  \Longleftrightarrow \forall \omega \in \cW,\  i \geq 1: \   m_{1,\omega}(\sigma,N) + \cdots + m_{i,\omega}(\sigma,N) \leq m_{1,\omega}(\sigma',N') + \cdots + m_{i,\omega}(\sigma',N').$$

\end{definition}

Denote by $I_v$ the inertia subgroup of the Weil group $W_{F_v}$ at $v$, and let $\cI$ denote the set of isomorphism classes of irreducible representations of $I_v$ with open kernel. 
For every $\theta \in \cI$, define $\sigma[\theta]$ to be the isotypic component of $\left.\sigma\right|_{I_v}$, whose irreducible subquotients are isomorphic to $\theta$. Since $N$ commutes with the image of $I_v$, these isotypic components are preserved by the monodromy operator, and thus we can define $N[\theta]$ as the restriction of $N$ to $V[\theta]$.

\begin{definition}\label{preci}
Let $(\sigma,N)$  be a Weil-Deligne representations of $W_{F_v}$ over $\overline{\bQ}_p$.  For each $\theta \in \cI$, we can define a unique decreasing sequence of non-negative integers $n_{1,\theta}(\sigma,N) \geq n_{2,\theta}(\sigma,N) \geq \hdots$ which determines the conjugacy class of the monodromy operator $N[\theta]$. It is a partition of the integer $\dim(r[\theta])/\dim(\theta)$. If $(\sigma',N')$ is another Weil-Deligne representation, then we define
\begin{eqnarray*}
(\sigma,N) \prec_I (\sigma',N') &\Longleftrightarrow&\left.\sigma\right|_{I_v} \cong \left.\sigma'\right|_{I_v} \quad \mbox{ and } \quad \forall \theta \in \cI,\ \  i \geq 1, \\
&& n_{1,\theta}(\sigma,N) + \hdots + n_{i,\theta}(\sigma,N) \leq n_{1,\theta}(\sigma',N') + \hdots + n_{i,\theta}(\sigma',N').
\end{eqnarray*}
\end{definition} 

We have the following lemma relating the two dominance relations defined above. For any sequence of integers $(m_i)_{i \in \bZ_{> 0}}$ and $d \in \bZ_{>0}$, let $d\cdot (m_i)_i$ be the sequence of integers $(m_1, m_1, \hdots, m_1,$ $m_2, m_2, \hdots, m_2, \hdots)$ where each $m_i$ occurs $d$ times.
\begin{lemma}[Lemma 6.5.3 in \cite{BC}]\label{chenevier}
Let $(\sigma,N)$ be a Weil-Deligne representation of $W_{F_v}$. 
\begin{enumerate}
\item Let $\omega \in \cW$ and $\theta$ an irreducible constituent of $\left.s\right|_{I_v}$ for any $s \in \omega$. Then $\sigma[s'] \cap \sigma[\theta] = 0$ if $s'$ is not an unramified twist of $s$. Furthermore, if $d = \dim(s)/\dim(\theta)$, then 
	$$(n_{1,\theta}(\sigma,N), n_{2,\theta}(\sigma,N), \hdots) = d \cdot (m_{1,\omega}(\sigma,N), m_{2,\omega}(\sigma,N), \hdots).$$ 
\item If $(\sigma',N')$ is another Weil-Deligne representation of $F_v$ such that $\sigma^{ss} \cong \sigma'^{ss}$, then $(\sigma,N) \prec (\sigma',N') \Leftrightarrow (\sigma,N) \prec_I (\sigma',N').$
\end{enumerate}
\end{lemma}

\noindent From Lemma \ref{chenevier} and Proposition \ref{localglobalss}, it suffices to prove that 
	$$\WD(\left.r_{p,\imath}(\Pi)\right|_{W_{F_v}})^{\Frob-ss} \prec_I \rec_{F_v}(\BC(\Pi_\ell)_v|\det|_v^{(1-2n)/2})$$ 
in order to conclude the proposition. We start by characterizing irreducible representations of $I_v$ with open kernel.

\begin{definition} If $(\theta,V)$ is a representation of $I_v$ and $\tau$ is an irreducible representation of a subgroup $H$ of $I_v$, set $(\theta[\tau],V[\tau])$ to be the $\tau$-isotypical component of the $H$-representation $(\left.\theta\right|_H,\left. V\right|_H)$. Furthermore, if $N$ is a commuting nilpotent endomorphism of $V$, then set $N[\tau] = N \cap V[\tau]$.
\end{definition}

Let $P$ denote a Sylow pro-$p$-subgroup of $I_v$. Recall that there is a map $t_p: I_v \rightarrow \bZ_p$ since $v \nmid p$ and let $I_v^p := \ker t_p$. Recall that there is also an identification of $P$ with $I_v/I_v^{p}$.
Let $\cI^p$ denote the set of isomorphism classes of representations of $I_v^p$ with open kernel; there is a canonical action on $\cI^p$ by $I_v/I_v^p$ acting by conjugation. Let $\cI^p_0$ denote the subset of elements of $\cI^p$ with open stabilizer in $I_v/I_v^p$. For $\eta \in \cI^p_0$, set $I^\eta = \Stab_{I_v}(\eta)= \{ i \in I_v: \eta(i \ast i^{-1}) \cong \eta\}$, which is open in $I_v$. Additionally, fix a choice of topological generator $g_{\eta}$ of $P \cap I^{\eta}$ such that $I^{\eta} = \overline{\langle I_v^p , g_{\eta}\rangle}.$ Note that $g_{\eta}$ has pro-$p$-order and can be chosen so that $g_{\eta(g \ast g^{-1})} = g_{\eta}$ for all $g \in P$.

\begin{lemma}\label{9.3} If $\eta \in \cI^p_0$, there exists an irreducible representation $\tilde{\eta}$ of $I^{\eta}$ with open kernel such that $\left.\tilde{\eta}\right|_{I^p_v} \cong \eta.$
\end{lemma}
\begin{proof}
Since $\eta \in \cI^p_0$, we have that $I^p_v/\ker(\eta)$ is finite order, and conjugation by $g_{\eta}$ induces an automorphism of the quotient. This automorphism must have finite order as well, and since $g_{\eta}$ has pro-$p$-order in $I_v$, conjugating by $g$ must have $p$-power order as an automorphism of $I^p_v/\ker(\eta)$. This implies that there is some nonnegative integer $n$ such that $g_{\eta}^{p^n}$ centralizes $I^p_v/\ker(\eta)$. Let $A_{g_{\eta}}$ be an invertible matrix such that $\eta(g_{\eta} \ast g_{\eta}^{-1}) = A_{g_{\eta}}\eta(\ast) A_{g_{\eta}}^{-1}$. Then $A_{g_\eta}^{p^n}$ centralizes $\eta$ and therefore must be a scalar since $\eta$ is irreducible; thus, we may suppose that $A_{g_{\eta}}^{p^n} = 1$. We can then define the representation 
	$$\tilde{\eta}: I^\eta \rightarrow \GL_{\dim \eta}(\overline{\bQ}_p) \qquad i_0 g_{\eta}^k \mapsto \eta(i_0)A_{g_{\eta}}^k \qquad \mbox{where } i_0 \in I_v^p.$$ 
Furthermore, since $\eta$ is irreducible, $\tilde{\eta}$ is also irreducible.
\end{proof}

For each $\eta \in \cI^p_0$, choose once and for all a lift $\tilde{\eta}$ to $I^{\eta}$ such that $\widetilde{\eta(g \ast g^{-1})} = \tilde{\eta}(g \ast g^{-1})$ for all $g \in P$.  
If $\eta \in \cI^p_0$ and $\chi$ is a character of $I^{\eta}$ with open kernel containing $I_v^p$, set $\theta_{\eta,\chi} := \Ind^{I_v}_{I^{\eta}} \tilde{\eta} \otimes \chi$.

\begin{lemma}\label{9.4}  If $\eta \in \cI^p_0$ and $\chi$ is a character of $I^{\eta}$ with open kernel containing $I^p_v$ then: 
\begin{enumerate}
\item $\theta_{\eta,\chi}$ is irreducible and 
$\left. \theta_{\eta,\chi}\right|_{I^\eta} \cong \displaystyle{\bigoplus_{[i] \in I_v/I^\eta}}\tilde{\eta}(i \ast i^{-1}) \otimes \chi.$
\item $\theta_{\eta,\chi} \cong \theta_{\eta',\chi'}$ if and only if $\chi = \chi'$ and $\eta' \cong \eta(i \ast i^{-1})$ for some $i \in I_v$.
\item Every irreducible representation of $I_v$ with open kernel arises in this way.
\end{enumerate}
\end{lemma}

\begin{proof} 1.  For any character $\chi: I^\eta \rightarrow \overline{\bQ}_p^\times$ with open kernel containing $I^p_v$, $\tilde{\eta} \otimes \chi$ is irreducible since $\eta$ is. 
Thus, we can prove that $\theta_{\eta,\chi}$ is irreducible using Mackey's Criterion: 
Consider some element $i \in I_v \smallsetminus I^\eta$ and define $c_i \in \End(I^\eta)$,
	$$c_i: x \mapsto ixi^{-1}.$$
We want to show that $\theta_{\eta,\chi}$ and $\theta_{\eta,\chi} \circ c_i$ are disjoint representations of $I^\eta$, i.e. have no irreducible component in common.  It is enough to see that they are disjoint on $I_v^p$. 
Since $\left.\theta_{\eta,\chi} \circ \id\right|_{I_v^p} = \eta$ and $\left.\theta_{\eta,\chi} \circ c_i\right|_{I_v^p} = \eta(i \ast i^{-1})$ for $i \notin I^{\eta}$, these are not isomorphic irreducible representations, thus they must be disjoint. The second part follows from Frobenius reciprocity and the definition of $\theta_{\eta,\chi}$ as an induced representation from the stabilizer of $\eta$ in $I_v$ to $I_v$.

2. Next, we prove that $\theta_{\eta,\chi}$ and $\theta_{\eta',\chi'}$ are isomorphic if and only if for some $i \in I_v$, $\eta(\ast) \cong \eta'(i \ast i^{-1})$ and $\chi = \chi'$. One direction follows from the first part of the lemma. To prove the converse, assume $\theta_{\eta,\chi}$ and $\theta_{\eta',\chi'}$ are isomorphic. Restricting to $I_v^p$, we have
	$$\bigoplus_{[i] \in I_v/I^\eta} \eta(i \ast i^{-1}) \cong \left.\theta_{\eta,\chi}\right|_{I_v^p} \cong \left.\theta_{\eta',\chi'}\right|_{I_v^p} \cong \bigoplus_{[i] \in I_v/I^{\eta'}} \eta'(i \ast i^{-1}),$$ 
thus $\eta \cong \eta'(i \ast i^{-1})$ for some $[i] \in I_v/I_v^p$. This further implies $I^\eta \cong I^{\eta'}$ where the isomorphism is given by conjugation by $i$ since for any element $g \in I^\eta$, 
	$$\eta'(igi^{-1} \ast ig^{-1}i^{-1}) \cong \eta(ig \ast g^{-1}i^{-1}) \cong \eta(i \ast i^{-1}) \cong \eta'(\ast).$$
In fact, since $I_v/I_v^p$ is abelian, we have proven that $I^\eta = I^{\eta'}$. 

It remains to show that $\chi \cong \chi'$. By Frobenius reciprocity, 
	$$\Hom_{I_v}(\theta_{\eta,\chi'},\theta_{\eta,\chi}) = \Hom_{I^{\eta}}(\tilde{\eta} \otimes \chi', \bigoplus_{[i] \in I_v/I^\eta}\tilde{\eta}(i \ast i^{-1}) \otimes \chi).$$
Since $\tilde{\eta}(i \ast i^{-1}) \otimes \chi$ is irreducible, it remains to check that $\tilde{\eta} \otimes \chi \not \cong \tilde{\eta}$ as representations of $I^\eta$ for nontrivial $\chi$. 
Let $\chi(g_{\eta}) = \lambda_{g_{\eta}}$ and note that if $\tilde{\eta}(g_{\eta}) = (\tilde{\eta}\otimes \chi)(g_\eta) = \lambda_{g_\eta} \tilde{\eta}(g_\eta)$, thus either $\lambda_{g_\eta} = 1$ or $\tr(\tilde{\eta}(g_{\eta}))$ is zero; however, since $\tilde{\eta}$ is irreducible, for any $h \in I^p_v$, we have $\tilde{\eta}(g_{\eta}h) = \lambda_{g_{\eta}}\tilde{\eta}(g_{\eta} h)$, and for some $h$, $\tr(\tilde{\eta}(g_{\eta}h)) \neq 0$. Thus, $tr(\tilde{\eta}(g_{\eta})) \neq 0$, and so we must have that $\lambda_{g_\eta} = 1$. Thus, we conclude that $\theta_{\eta,\chi} \neq \theta_{\eta',\chi'}$ when $\chi \neq \chi'$ or $\eta$ and $\eta'$ are not in the same orbit of $\cI^p_0$ under the action of $I_v/I_v^p$ (or equivalently, $I_v/I^\eta$).

3. Finally, we show that any irreducible (finite-dimensional) representation of $I_v$ arises as $\theta_{\eta,\chi}$ for some $\eta$ and $\chi$. Let $\theta: I \rightarrow \GL(V)$ be an irreducible representation, and restrict to $I_v^p$. Let $\oplus_{\eta \in \cI^p_0} V[\eta]$ denote the decomposition of $\left.\theta\right|_{I_v^p}$ into its isotypic components. For each $\eta$, $I^{\eta} = \Stab_I(\eta)$ acts on $V[\eta]$ and furthermore, each $i \in I$ induces an identification of $V[\eta]$ and $V[\eta(i \ast i^{-1})]$. This implies that $\Ind_{I^{\eta}}^I V[\eta] \cong V$, and thus as a representation of $I^{\eta}$, $V[\eta]$ is irreducible. There is an isomorphism as $\overline{\bQ}_p$-vector spaces
\begin{eqnarray}\label{isom}
	\Hom_{I^p_v}(\left.\tilde{\eta}\right|_{I^p_v},\left.V[\eta]\right|_{I^p_v}) \otimes \tilde{\eta} \stackrel{\sim}{\longrightarrow} V[\eta].
\end{eqnarray}
The space $\Hom_{I^p_v}(\left.\tilde{\eta}\right|_{I^p_v},\left.V[\eta]\right|_{I^p_v})$ has an action of $i \in I^{\eta}$ by conjugation, and $I^p_v$ acts trivially. With this action, (\ref{isom}) is indeed an isomorphism of $I^{\eta}$-representations. However, since $V[\eta]$ is irreducible, $\Hom_{I^p_v}(\left.\tilde{\eta}\right|_{I^p_v},\left.V[\eta]\right|_{I^p_v})$ must be irreducible over $I^{\eta}/I^p_v$, which is abelian. Letting $\chi = \Hom_{I^p_v}(\left.\tilde{\eta}\right|_{I^p_v},\left.V[\eta]\right|_{I^p_v})$, we conclude that $\theta = \theta_{\eta,\chi}$.
\end{proof}

We now consider a more useful version of Definition \ref{preci} to all representations of $I_v$ with open kernel and commuting nilpotent endomorphism.

\begin{proposition} If $(\sigma,V,N)$ and $(\sigma',V',N')$ are two Weil-Deligne representations, then $(\sigma,V,N) \prec_I (\sigma',V',N)$ if and only if $\left.\sigma\right|_{I_v} \cong \left.\sigma'\right|_{I_v}$
	$$\dim(\ker (N^j ) \cap V[\theta_{\eta,\chi}]) \geq \dim(\ker ({N'}^j)\cap V'[\theta_{\eta,\chi}])$$ 
for all $j \in \bZ_{>0}$, $\eta \in \cI^p_0$, and $\chi$ a character of $I^\eta/I^p_v$ with open kernel.  \end{proposition}

\begin{proof} Note that for any $\theta \in \cI$, the conjugacy class of $N[\theta]$ (resp. $N'[\theta]$) is determined by the partition of $\dim(\sigma[\theta])/\dim(\theta)$ (resp. $\dim(\sigma'[\theta])/\dim(\theta)$) given by $(n_{i,\theta}(\sigma,N))_{i \geq 1}$ (resp. $(n_{i,\theta}(\sigma',N'))_{i \geq 1}$). The condition
	$$n_{1,\theta}(\sigma,N) + \hdots + n_{i,\theta}(\sigma,N) \leq n_{1,\theta}(\sigma',N') + \hdots + n_{i,\theta}(\sigma',N')$$
is equivalent to the condition
	$$\rk N[\theta]^j \leq \rk (N'[\theta])^j \qquad \forall j \geq 0.$$
Since we require $\left.\sigma\right|_{I_v} \cong \left.\sigma'\right|_{I_v}$ in both definitions, we have that their dimensions are equal, thus $\rk N[\theta]^j \leq \rk (N'[\theta])^j$ is equivalent to 
	$$\dim \ker N[\theta]^j \geq \dim \ker N'[\theta]^j$$
By Lemma \ref{9.4}, we know that all $\theta \in \cI$ are of the form $\theta_{\eta,\chi}$ where $\eta \in \cI^p_0$ and $\chi$ is a character of $I^{\eta}$ with open kernel containing $I^p_v$ and so we are done.
\end{proof}
Furthermore, given $j \in \bZ_{>0}$, $\eta \in \cI^p_0$, and $\chi$ a character of $I^{\eta}/I^p_v$ with open kernel, then using the fact that $\dim \ker N = [I:I^\eta]\dim \ker \left. N \right|_{\theta_{\eta,\chi}[\tilde{\eta} \otimes \chi]}$ (coming from the Lemma \ref{9.4}(1)), 
we can conclude
	$$\dim(\ker N^j \cap V[\theta_{\eta,\chi}]) \geq \dim(\ker {N'}^j \cap V'[\theta_{\eta,\chi}])
 \Leftrightarrow  \dim(\ker N^j \cap V [\tilde{\eta} \otimes \chi]) \geq \dim (\ker {N'}^j \cap V'[\tilde{\eta} \otimes \chi])$$

If $\eta$ denotes a representation of $I_v^p$ with open kernel and $f:I_v^p \rightarrow \overline{\bQ}_p$ is a locally constant function, then let $\eta(f) := \int_{I_v^p} f(i)\eta(i)di,$ where $di$ denotes the Haar measure on $I_v^p$ (normalized so that vol($I_v^p$) = 1). Since $I^p_v$ is compact, this integral is in fact a finite sum. Recall that for each $\eta$, we fixed a choice of topological generator $g_{\eta}$ of $P \cap I^{\eta}$ such that $I^{\eta} = \overline{\langle I_v^p , g_{\eta}\rangle}.$ The following lemma describes the existence of projection operators for representations of $I_v^p$ and the relationship between the image of $\tilde{\eta}$ and $\eta$ (both irreducible).

\begin{lemma} If $(\eta,V)$ and $(\eta',V') \in \cI^p_0$, then: \begin{enumerate} \item There exists a locally constant function $\epsilon_{\eta}: I_v^p \rightarrow \overline{\bQ}_p$ sending $i \mapsto \frac{\tr(\eta^{\ast}(i))}{\dim \eta}$ (where $\eta^{\ast}$ denotes the dual representation) such that $\eta(\epsilon_\eta) = 1$ but $\eta'(\epsilon_\eta) = 0$ for all $\eta' \not \cong \eta$. 
\item There exists a locally constant function $a_{\eta}: I_v^p \rightarrow \overline{\bQ}_p$ such that $\tilde{\eta}(g_{\eta}) = \eta(a_{\eta})$ but $\eta'(a_{\eta}) = 0$ if $\eta \not\cong \eta'$.
\end{enumerate}
\end{lemma}

\begin{proof} The first part is clear. As for the second part, since $\eta$ has open kernel, there is a finite quotient $I^v_p/\ker(\eta)$ through which it factors. Furthermore, $\eta$ is irreducible, and thus the matrix $\tilde{\eta}(g_{\eta}) \in \Hom(V)$ can be written as a sum $\eta(a_{\eta}) = \sum_{g \in I^v_p/\ker(\eta)} \eta(g) a_{\eta}(g)$ where $h$ is a uniquely determined sum of matrix coefficients. By orthogonality, we have that $\eta'(a_{\eta}) = 0$ for $\eta' \not \equiv \eta$. Recall that the Peter-Weyl theorem (see e.g.,\cite{simon}) gives an isomorphism 
$$\Hom(I^v_p/\ker(\eta),\overline{\bQ}_p) \stackrel{\sim}{\longrightarrow} \bigoplus_{(r,V) \in \Irr(I^v_p/\ker(\eta))} \End_{\overline{\bQ}_p}(V),$$
and thus $a_{\eta}$ pulls back to a locally constant function of $I^v_p$.

\end{proof}

For each $\eta \in \cI^p_0$, fix a choice of $\epsilon_{\eta}$ and $a_{\eta}$ as described in the above lemma. If $(\sigma,V,N)$ (resp. $(\sigma',V',N')$) uniquely determine (local) Galois representations $\rho_v$ (resp. $\rho_v'$) of $G_{F_v}$ acting on the same underlying vector space $V$ (resp. $V'$), then recall that the defining relation between $\rho_v$ and $(\sigma,V,N)$ is 
	$$\rho_v(i) = r(i) \exp(t_p(i)N) \qquad \mbox{ for } i \in I_v.$$
If $i \in I_v$ is an element such that $t_p(g)$ is nonzero, then we can write $\log (r(i)^{-1}\rho_v(i)) = t_p(i)N$. Additionally, for all positive $j$, $\rk(t_p(i)N)^j = \rk N^j$ and for any unipotent matrix $U$, $\rk(\log U)^j = \rk(U - 1)^j$, and thus
	$$\rk\left( r(g)^{-1} \rho_v(g) - \id\right)^j = \rk N^j.$$ This implies that $$\rk(\left.N\right|_{V[\tilde{\eta}\otimes\chi]})^j = \rk(\left.(\rho_v(g_{\eta}) - r(g_{\eta}))^j\right|_{V[\tilde{\eta}\otimes\chi]}),$$ and we have  that $(\sigma,N) \prec_I (\sigma',N')$ if and only if for all $j \in \bZ_{>0}$, $\eta \in \cI^p_0$, and $\chi$ a character of $I^{\eta}/I^p_v$ with open kernel 
	$$\dim(\ker \left.(\rho(g_{\eta}) - r(g_{\eta}))^j\right|_{V[\tilde{\eta} \otimes \chi]}) \geq  \dim(\ker \left.(\rho'(g_{\eta}) - r'(g_{\eta}))^j\right|_{V'[\tilde{\eta} \otimes \chi]}).$$ Additionally, since $\ker \left.(\rho(g_{\eta}) - r(g_{\eta}))^j\right|_{V[\tilde{\eta} \otimes \chi]} = \ker(\rho(g_{\eta}) - \rho(a_{\eta})\chi(g_{\eta}))^j,$ we can then conclude:

\begin{lemma}\label{techdefn} If $(\rho,V),(\rho,V')$ are two continuous $2n$-dimensional representations of $I_v$ (arising from continuous $G_{F_v}$-representations), then $(\rho,N) \prec_I (\rho',N')$ if and only if $\left. \rho\right|_{I_v^p} \cong \left.\rho'\right|_{I_v^p}$ and 
	$$\wedge^k(\rho'(g_{\eta}) - \rho'(a_{\eta})\zeta)^j = 0 \Rightarrow\wedge^k(\rho(g_{\eta}) - \rho(a_{\eta})\zeta)^j = 0$$
for all $j \in \bZ_{>0}$, $\eta \in \cI^p_0$, and $p$-power root of unity $\zeta$. 
\end{lemma}
\begin{proof}
This follows from the above and the fact that for any $A \in \End(V)$, $\dim \ker A = \dim V + 1 - \min\{k \in \bZ_{>0} : \wedge^k A = 0\}.$
\end{proof}

Suppose $\rho_v'$ is a  local $p$-adic $G_{F_v}$-Galois representation of dimension $2n$, and $\rho$ is a semisimple continuous $2n$-dimensional global Galois representations of $G_{F} \supset I_v$. Then $\wedge^k \rho$ is also semisimple, and $\WD(\left.\rho\right|_{G_{F_v}})^{\Frob-ss} \prec_I \WD(\rho_v')^{\Frob-ss}$ if and only if for all $j,k \in \bZ_{>0}$, $\eta \in \cI^p_0$, $\zeta$ a $p$-power root of unity,
	$$\wedge^k (\rho'(g_{\eta}) - \rho'(a_{\eta})\zeta)^j = 0 \quad \Rightarrow \quad \tr (\wedge^k (\rho(g_{\eta}) - \rho(a_{\eta})\zeta)^j \rho(\tau)) = 0 \quad \forall \tau \in G_F$$
because trace is a non-degenerate bilinear form on the image of semisimple representation. For any $\rho$, we can extend it by linearity to $\rho: \overline{\bQ}_p[G_F] \rightarrow \GL(V)$, and let $b_{\eta,\zeta} := g_{\eta} - \zeta \cdot a_{\eta} \in \overline{\bQ}_p[G_F]$. Then $\WD(\rho)^{\Frob-ss} \prec_I \WD(\rho')^{\Frob-ss}$ if and only if for all $k,j \in \bZ_{>0}$, $\eta \in \cI^p_0$, and $p$-power roots of unity $\zeta$, 
	$$\wedge^k (\rho'(g_{\eta}) - \rho'(a_{\eta})\zeta)^j = 0  \quad \Rightarrow \quad \tr \wedge^k \rho(b_{\eta,\zeta}^j \tau) = 0  \quad \forall \tau$$
If $T$ denotes a $2n$-dimensional continuous pseudocharacter of $G_{F}$, then by extending linearly and 
using the recursive formula for a matrix $A$, $\tr \wedge^k A = \frac{1}{k} \sum_{m=1}^k (-1)^{m-1} \tr(A^m) \tr \wedge^{k-m}(A)$, we can define 
$$\wedge^k T: \overline{\bQ}_p[G_{F}] \rightarrow \overline{\bQ}_p \qquad g \mapsto \frac{1}{k} \sum_{m=1}^k (-1)^{m-1} T(g^m) \wedge^{k-m}T(g)$$ for $k \leq 2n$. In the sequel, we will be interested in whether the following function
	$$B^{k,j}_{\eta,\zeta}(T): G_F \rightarrow \overline{\bQ}_p \quad \tau \mapsto \wedge^k T(b_{\eta,\zeta}^j \tau)$$
is identically zero.
	
	\subsection{Proof of $\prec$}
In this section we prove Proposition \ref{precthm}.
\begin{proof}
Fix $\ell \in S^{\spl}$ and let $v \mid \ell$ be a prime of $F$ in $\cS^{\spl} \sqcup \overline{\cS}^{\spl}$. We have already seen that for $\Pi$ satisfying the hypothesis of the proposition, 
	$$\WD(\left.r_{p,\imath}(\Pi)\right|_{W_{F_v}})^{ss} \cong \rec_{F_v}(\BC(\Pi_\ell)_v |\det|_v^{(1-2n)/2})^{ss}.$$
By Lemma \ref{chenevier}, it therefore remains to show that 
	$$\WD(\left.r_{p,\imath}(\Pi)\right|_{W_{F_v}})^{\Frob-ss} \prec_{I} \rec_{F_v}(\BC(\Pi_\ell)_v |\det|_v^{(1-2n)/2}).$$ 
For ease of notation, let the $p$-adic local Galois representation associated to the Frobenius semisimple Weil-Deligne representation $\rec_{F_v}(\BC(\Pi_\ell)_v |\det|_v^{(1-2n)/2})$ be denoted $\rho_{\Pi,v}^{\rec}$. We want to show that for all $\eta \in \cI^p_0$, $p$-power roots of unity $\zeta$, and $j,k \in \bZ_{>0}$,  
	$${\wedge}^k \rho_{\Pi,v}^{\rec}(b_{\eta,\zeta}^j) = 0 \quad \Rightarrow \quad  {\wedge}^k r_{p,\imath}(\Pi)(b_{\eta,\zeta}^j) = 0.$$ 
Recall that for $T_{\Pi} := \varphi_{\Pi} \circ T$ constructed in the proof of Proposition \ref{localglobalss}, there is a function $B^{k,j}_{\eta,\zeta}$ for each $j,k \in \bZ_{>0}$, $\eta \in \cI^p_0$, and $p$-power root of unity $\zeta$ such that 
$$B^{k,j}_{\eta,\zeta}(T_{\Pi})(\tau) = \tr \wedge^k (r_{p,\imath}(\Pi)(\epsilon_{\eta} g_{\eta}) - r_{p,\imath}(\Pi)(a_{\eta})\zeta)^j r_{p,\imath}(\Pi)(\tau).$$
By Lemma \ref{techdefn}, we want to show that  
	$${\wedge}^k \left(\rho_{\Pi,v}^{\rec}(g_\eta) - \zeta \cdot \rho_{\Pi,v}^{\rec}(a_{\eta})\right)^j = 0 \quad \Rightarrow \quad B^{k,j}_{\eta,\zeta}(T_\Pi) = 0.$$

Let $\cB_v$ denote the Bernstein component containing $\BC(\Pi_\ell)_v.$ 
 By Proposition 6.2 in \cite{SZ}, associated to $\Pi$, 
there exists an idempotent $e_{\Pi,\cB_v}$ inside the Bernstein center $\fz_{\cB_v}$ associated to $\cB_v$ such that:
\begin{itemize}
\item $e_{\Pi,\cB_v}(\BC(\Pi_\ell)_v) \neq 0$ 
\item $e_{\Pi,\cB_v}(\Pi_0) \neq 0 \Rightarrow \rec(\Pi_0) \prec_I \rec(\BC(\Pi_\ell)_v)$ for all irreducible $\Pi_0$ of $GL_{2n}(F_v)$, 
\end{itemize}
If $e_{\Pi}$ denotes the image of $e_{\Pi,\cB_v} \in \fz_{\fB_v}$, where $\fB_v$ is the disjoint union of Bernstein components containing $\cB_v$ defined in \S\ref{ramified}, let $\tilde{e}_{\Pi} := d(e_{\Pi})e_{\Pi} \in \fz_\ell^0$ and abusing notation, let $\tilde{e}_{\Pi}$ also denote its own image in $\cH^{\spl}_{\bZ_p}$ and $\End(H^0(\cX^{\min}_U, \cE^{\sub}_{U,\rho_{\underline{b}}}))$ for any $\underline{b} \in X^\ast(T_{n/\overline{\bQ}_p})_{\cl}^+$ and any neat open compact $U$ such that $U_v = K_{\fB_v}$. 

\begin{lemma} Let $\underline{b} \in X^\ast(T_{n/\overline{\bQ}_p})_{\cl}^+$ and let $\bT^p_{U,\underline{b}}$ denote the image in $\cH^p_{\bZ_p}$ in $\End_{\bZ_p}(H^0(\cX^{\min}_{U}, \cE^{\sub}_{U,\underline{b}}))$ where $U_\ell$ satisfies (\ref{compact}) for every $\ell \in S_{\spl}$. There is a continuous representation $r_{\underline{b}}:G_F^S \rightarrow \GL_{2n}(\bT_{\underline{b}} \otimes \overline{\bQ}_p)$ described in (\ref{heckerep}) for every $\underline{b}$, and let $T_{\underline{b}} = \tr r_{\underline{b}}$.
Assume that $\eta \in \cI^p_0$, $\zeta$ a $p$-power roof of unity, and $k,j \in \bZ_{>0}$ are such that 
	$${\wedge}^k \left(\rho_{\Pi,v}^{\rec}(g_\eta) - \zeta\cdot \rho_{\Pi,v}^{\rec}(a_\eta)\right)^j = 0.$$ For each ${\underline{b}}$, the map $\tilde{e}_{\Pi}B^{k,j}_{\eta,\zeta}(T_{\underline{b}}): G_F^S \rightarrow \bT_{\underline{b}}$ is identically zero.	
\end{lemma}
\begin{proof} Recall that $\bT_{\underline{b}} \cong \oplus_{\Pi_0} \overline{\bQ}_p$ where the sum runs over irreducible admissible representations of $G(\bA^{p,\infty} \times \bZ_p)$ with $\Pi_0^U \neq (0)$ which occur in $H^0(\cX^{\min},\cE^{\sub}_{\underline{b}})$. We will prove that for each $\Pi_0$, the composition 
	$$\varphi_{\Pi_0} \circ \tilde{e}_{\Pi}B^{k,j}_{\eta,\zeta}(T_{\underline{b}}): G^S_F \rightarrow \bT_{\underline{b}} \stackrel{\varphi_{\Pi_0}}{\rightarrow} \overline{\bQ}_p$$ is zero. Assume $\tilde{e}_{\Pi}(\BC(\Pi_{0,\ell})_v) \neq 0$ for some $\Pi_0 \in H^0(\cX^{\min},\cE^{\sub}_{\underline{b}})$. Then $\rec(\BC(\Pi_{0,\ell})_v) \prec_I \rec(BC(\Pi_{\ell})_v)$, 
and so by Lemma \ref{techdefn},  for $\eta \in \cI^p_0$
 $${\wedge}^k \left(\rho_{\Pi,v}^{\rec}(\epsilon_\eta g_\eta) - \zeta\cdot \rho_{\Pi,v}^{\rec}(a_\eta)\right)^j = 0 \Rightarrow
{\wedge}^k \left(\rho_{\Pi_0,v}^{\rec}( g_\eta) - \zeta \cdot\rho_{\Pi_0,v}^{\rec}(a_\eta)\right)^j = 0.$$
By Corollary \ref{localglobalss} and the above corollary, we know this implies that 
	$${\wedge}^k \left(r_{p,\imath}(\Pi_0)(g_\eta) - \zeta \cdot r_{p,\imath}(\Pi_0)(a_\eta)\right)^j = 0,$$
thus $\varphi_{\Pi_0} \circ B_{\eta,\zeta}^{k,j}(T_{\underline{b}}) = 0$.
\end{proof}

Since $T_{\Pi}$ is constructed in terms of $T_{\underline{b}}$, if $\tilde{e}_\Pi(\BC(\Pi_{\ell})_v)B_{\eta,\zeta}^{k,j}(T_{\underline{b}})$ is identically zero for all $\underline{b} \in X^{\ast}(T_{n/\overline{\bQ}_p})$, 
then $\tilde{e}_{\Pi}B_{\eta,\zeta}^{k,j}(T_{\Pi})$ is also identically zero. Since $\tilde{e}_{\Pi}(\BC(\Pi_{\ell})_v) \neq 0$, we can conclude that $B_{\eta,\zeta}^{k,j}(T_\Pi) = 0$ if $\eta \in \cI^p_0$, $\zeta$ a $p$-power root of unity, and $k,j \in \bZ_{>0}$ are such that 
	$${\wedge}^k \left(\rho_{\Pi,v}^{\rec}(g_\eta) - \zeta \cdot \rho_{\Pi,v}^{\rec}(a_\eta)\right)^j = 0.$$
This implies that $\WD(\left.r_{p,\imath}(\Pi)\right|_{W_{F_v}})^{\Frob-ss} \prec_{I} \rec_{F_v}(\BC(\Pi_\ell)_v |\det|_v^{(1-2n)/2})$. Thus, we conclude
	$$\WD(\left.r_{p,\imath}(\Pi)\right|_{W_{F_v}})^{\Frob-ss} \prec \rec_{F_v}(\BC(\Pi_\ell)_v |\det|_v^{(1-2n)/2}).$$
	\end{proof}
\section{Group Theory}
Let $\Gamma$ be a topological group and let $\frak{F}$ be a dense set of elements of $\Gamma$.  Let $k$ be an algebraically closed, topological field of characteristic 0 and let $d \in \bZ_{>0}$.
	Let $\mu: \Gamma \rightarrow k^\times$ 
be a continuous homomorphism such that $\mu(f)$ has infinite order for all $f \in \mathfrak{F}$. For $f \in \mathfrak{F}$ let $\cE^1_f$ and $\cE^2_f$ be two $d$-elements multi set of elements of $k^\times$. Let $\cM$ be an infinite subset of $\bZ$. For $m\in\cM$ suppose that 
	$$\rho_m: \Gamma \rightarrow \GL_{2d}(k)$$
be a continuous semi-simple representation such that for every $f \in \mathfrak{F}$ the multi-set of roots of the characteristic polynomial of $\rho_m(f)$ equals 
	$$\cE^1_f \bigsqcup \cE^2_f \mu(f)^m.$$
\begin{proposition}[Proposition 7.12 in \cite{HLTT}] There are continuous semi simple representations
	$$\rho^i: \Gamma \rightarrow \GL_d(k)$$
for $i = 1,2$ such that for all $f \in \mathfrak{F}$ the multi set of roots of the characteristic polynomial of $\rho^i(f)$ equals $\cE^i_f$.
\end{proposition}

\begin{theorem} Suppose that $\pi$ is a cuspidal automorphic representation of $\GL_n(\bA_F)$ such that $\pi_\infty$ has the same infinitesimal character as an algebraic representation of $\RS^F_{\bQ} \GL_n$. Then there is a continuous semi simple representation
	$$r_{p,\imath}(\pi): G_F \rightarrow \GL_n(\overline{\bQ}_p)$$
such that if $v \nmid p$ is a prime of $F$ which is split over $F^+$, then 
	$$\left.\WD(r_{p,\imath}(\pi)\right|_{W_{F_v}})^{\Frob-ss} \prec \imath^{-1}\rec_{F_v}(\pi_v|\det|_v^{(1-n)/2}).$$
In particular, if $\pi$ and $F$ are unramified at $v$, then $r_{p,\imath}(\pi)$ is unramified.
\end{theorem}
\begin{proof} Assume that $n > 1$. Let $\cS_{\spl}$ denote the set of primes $v$ of $F$ at which $\pi$ is ramified and does not split over $F^+$; and let $G_{F,S}$ denote the Galois group over $F$ of the maximal extension of $F$ unramified outside $S$. Let $\Gamma = G_{F,S}$ and $k = \overline{\bQ}_p$, and $\mu = \epsilon_p^{-2}$, and $\cM$ consisting of all sufficiently large integers $m$, and 
	$$\rho_m = R_{p,\imath}(\pi,m)= R_p\left(\imath^{-1} \Ind_{P_{(n)}^+(\bA^{p,\infty})}^{G(\bA^{p,\infty})} (\pi^\infty||\det||^m \times 1)\right) \otimes \epsilon_p^{-m} \qquad m \in \cM$$ 
	For each $v$ and let $k(v)$ denote the residue field of $F_v$. Let $\fF$ contain all elements $\sigma_v \in W_{F_v}$ which projects to a power of Frobenius under the map $W_{F_v} \rightarrow \Gal(\overline{k(v)}/k(v))$, where $v \notin \cS_{\spl} \cup \{p\}$. Denote by $\sigma_{{}^cv}$ the image of $\sigma_v$ under the isomorphism $W_{F_v} \cong W_{F_{{}^cv}}$ induced by conjugation $c$. Define $\cE^1_{\sigma_v}$ to be the multiset of roots of the characteristic polynomial $\imath^{-1}\rec_{F_v}(\pi_v|\det|_v^{(1-n)/2})(\sigma_v)$ and $\cE^2_{\sigma_v}$ equal to the multiset of roots of the characteristic polynomial of $\imath^{-1}\rec_{F_{{}^cv}}(\pi_{{}^cv}|\det|_{{}^cv}^{(-1+3n)/2})(\sigma_{{}^cv}^{-1}).$ We can then conclude 
	$$(\left.r_{p,\imath}(\pi)\right|_{W_{F_v}})^{ss} \cong \imath^{-1} \rec_{F_v}(\pi_v|\det|_v^{(1-n)/2}).$$

 Now, from the above and the density of $\fF$ in its proof, we have that for a sufficiently large integer $M$ in the sense of Corollary 11.1,
	$$R_p(\pi,M) \cong r_{p,\imath}(\pi) \oplus r_{p,\imath}({}^c\pi)^{c,\vee} \otimes \epsilon_p^{1-2n-2M}.$$ We will conclude $\left.\WD(r_{p,\imath}(\pi)\right|_{G_{F_v}})^{\Frob-{ss}} \prec \rec_{F_v}(\pi_v|\det|_v^{(1-n)/2})$ by comparing the isotypic components of each of the Weil-Deligne representations. Let $\cW_r$ denote the subset of $\cW$ (as defined in Definition \ref{prec}) containing equivalence classes of irreducible $n$-dimensional representations of $W_{F_v}$ over $\overline{\bQ}_p$ with open kernel that have nontrivial isotypic component on $\WD(\left.r_{p,\imath}(\pi)\right|_{G_{F_v}})^{\Frob-ss}$. Note that for large enough $M$, the isotypic components of $r_{p,\imath}(\pi)$ and $r_{p,\imath}({}^c\pi)^{c,\vee} \otimes \epsilon_p^{1-2n-2M}$ are completely disjoint. For such $M$, if $\Pi^M = \Ind_{P^+_{(n)}(\bA^{p,\infty})}^{G(\bA^{p,\infty})}(\pi^{\infty}||\det||^M \times 1)$, then $R_p(\Pi^M,M) = R_p(\Pi^M) \otimes \varepsilon^{-M}$ and so 
	$$\bigoplus_{\omega \in \cW_r} R_p(\Pi^M,M)[\omega] = r_{p,\imath}(\pi)$$
and
	$$\bigoplus_{\omega \in \cW_r} \rec_{F_v}(\BC(\Pi^M_{\ell})_v|\det|_v^{(1-2n-M)/2})[\omega] = \rec_{F_v}(\pi_v|\det|_v^{(1-n)/2}),$$
Since the definition of $\prec$ is component-by-component (indexed by elements of $\cW$), by Corollary 11.1 we can conclude the theorem.
\end{proof}

\begin{corollary} Suppose that $E$ is a totally real or CM field and that $\pi$ is a cuspidal automorphic representation such that $\pi_\infty$ has the same infinitesimal character as an algebraic representation of $\RS^{E}_{\bQ} \GL_n$. Then there is a continuous semi simple representation
	$$r_{p,\imath}: G_E \rightarrow \GL_n(\overline{\bQ}_p)$$
such that, if $q \neq p$ is a prime and if $v \mid q$ is a prime of $E$, then 
	$$\WD(\left.r_{p,\imath}(\pi)\right|_{W_{E_v}})^{\Frob-ss} \prec \imath^{-1} \rec_{E_v}(\pi_v|\det|_v^{(1-n)/2}).$$
\end{corollary}
\begin{proof} This can be deduced from Theorem 7.13 by using Lemma 1 of \cite{So} using the same argument as in Theorem VII.I.9 of \cite{HT}.
\end{proof}
	
\subsection*{Acknowledgements}
The author would like to thank Richard Taylor for suggesting this problem and for his invaluable assistance throughout this project; she would also like to thank Gaetan Chenevier, David Geraghty, and Jay Pottharst for answering questions. The author was supported by a National Defense Science \& Engineering Fellowship and an NSF Graduate Research Fellowship.

\end{document}